\documentclass[11pt,twoside]{article}

\pdfoutput=1
\usepackage[english]{babel}
\usepackage[utf8]{inputenc}

\usepackage{amsmath,amsthm,amssymb,tikz, lmodern,graphicx, caption,afterpage,pdfsync,epsfig,epstopdf,url,amsmath,amsfonts,amssymb,
subcaption,pgfplots,enumitem, stmaryrd, float, hyperref}
\pgfplotsset{compat=newest}
\usepgfplotslibrary{patchplots}
\usetikzlibrary{decorations.markings,calc,fadings,decorations.pathreplacing}

\usepackage{blindtext}
\usepackage{geometry}
\theoremstyle{plain}
\newtheorem{thm}{Theorem}[section]
\newtheorem*{thm-non}{Theorem}

\newtheorem{lem}[thm]{Lemma}

\newtheorem{defn}[thm]{Definition}
\newtheorem{rem}[thm]{Remark}

\newtheorem{prop}[thm]{Proposition} 

\numberwithin{equation}{section}

\newtheorem{hyp}{Hypothesis}

\def\xeta{X^{+\eta}}

\def\var#1#2{\mu_{(#1,#2)}}
\def\vardis#1#2{\boldsymbol{\mu}_{(#1,#2)}}

\def\dX{\partial X}
\def\R{\mathbb{R}}

\def\cH{\mathcal{H}}
\def\mdiv{\mathrm{div}}
\def\divx{\mathrm{div}_X}

\def\beps{{\boldsymbol\epsilon}}
\def\oX{{\text{vol}_X}}
\def\Nc{\mathcal{N}}

\def\Hh{\mathcal{H}}
\def\Bb{\mathcal{B}}
\def\Yy{\mathcal{Y}}

\newcommand{\PP}{\mathbb{P}}

\renewcommand{\emptyset}{\varnothing}
\newcommand{\Haus}{\mathcal{H}^2}
\newcommand{\tri}{\mathcal{T}}
\newcommand{\normM}{\ensuremath\boldsymbol{n}_X}
\newcommand{\bfn}{\ensuremath\boldsymbol{n}}
\newcommand{\tin}{\tri^{\operatorname{in}}}
\newcommand{\tout}{\tri^{\operatorname{out}}}

\newcommand{\Var}{{\rm Var}}
\newcommand{\test}{C_0^1(\R^3\times G(3,2)\times \R)}
\newcommand{\diam}{{\rm diam}}

\newcommand{\ie}{{\itshape i.e.}}
\newcommand{\abs}[1]{\left\|#1\right\|}
\DeclareMathOperator{\Ext}{Ext}
\newcommand{\res}{\mathop{\hbox{\vrule height 7pt width .5pt depth 0pt
\vrule height .5pt width 6pt depth 0pt}}\nolimits}
\newcommand\rst[2]{{#1}_{\restriction_{#2}}}
\newcommand{\normvar}[1]{\|#1\|}
\newcommand{\normvarW}[1]{\|#1\|_{W'}}

\title{The matching problem between functional shapes via a $BV$ penalty term: a $\Gamma$-convergence result}

\author{G. Nardi$^{1}$, B. Charlier$^{2}$,  A. Trouv\'e$^{1}$ }

\date{}

\begin{document}

\maketitle

\footnotetext[1]{Centre Borelli, Ecole Normale Supérieure de Paris-Saclay.
Emails:  giacomo.nardi, alain.trouve@ens-paris-saclay.fr}
\footnotetext[2]{IMAG, Universit\'e de Montpellier. Email:  benjamin.charlier@umontpellier.fr}

\begin{abstract}

The matching problem often arises in image processing and involves finding a correspondence between similar objects. In particular, variational matching models optimize suitable energies that evaluate the dissimilarity between the current shape and the relative template. A penalty term often appears in the energy to constrain the regularity of the solution.
To perform numerical computation, a discrete version of the energy is defined. Then, the question of consistency between the continuous and discrete solutions arises.

This paper proves a $\Gamma$-convergence result for the discrete energy to the continuous one. In particular, we highlight some geometric properties that must be guaranteed in the discretization process to ensure the convergence of minimizers.

We prove the result in the framework introduced in \cite{ABN}, which studies the matching problem between geometric structures carrying on a signal (fshapes). The matching energy is defined for $L^2$ signals and evaluates the fshapes difference in terms of varifolds norm. 
This paper maintains a dual attachment term, but we consider a $BV$ penalty term in place of the original $L^2$ norm.

\end{abstract}
\section{Introduction}\label{Introduction}

\paragraph{Context and previous work.} 
This paper discusses some theoretical aspects of the matching problem, which has several applications in image processing. The matching problem seeks a bijection between a current surface (or curve) and a target one by minimizing a dissimilarity function called matching energy. Such an error function is the sum of two terms: the penalty term defining the regularity of the optimal solution and the attachment term estimating the dissimilarity between current and target surfaces.

The solution to the matching problem defines the geometric transformation that links the two surfaces, enabling their comparison or transformation to a standard template.
Moreover, this variational approach is implemented in the discrete framework using the stepwise descent algorithm, which provides the evolution related to the mentioned geometric transformation.
The image processing community broadly studies the matching problem to define robust methods for shape analysis or image registration \cite{Rigid-evol, Geod-bv2, Bauer2020ANF, ABN, Pennec2019}.

We point out, in particular, the increasing use of tools from geometric measure theory (currents, varifolds) for the definition of matching energies \cite{ABN, Roussillon2016, Roussillon2020, Kaltenmark2017}. F. Almegren has developed the theory of varifolds \cite{Almgren}, afterward generalized by W.K. Allard \cite{Allard}, to generalize the notion of manifold in the framework of measure theory. Varfiold's theory considers a shape as a rectifiable measure and enables the definition of attachment terms in the sense of measures. This theoretical framework is suitable for obtaining compactness and existence results, but the computation of weak metrics remains a challenging problem. 

Nowadays, new developments in non-invasive acquisition techniques, such as Magnetic Resonance Imaging (MRI) or Optical Coherence Tomography (OCT), provide geometric-functional data for several diseases (e.g., cortical thickness in the study of Alzheimer's disease or thickness of retina layers for the evolution of glaucoma). Then, new methods are needed to deal with data containing geometrical and functional information \cite{LEE2017570, Aston2020} to describe anatomical variability and produce statistical analysis for related diseases.

In this context,  \cite{ABN}  introduces a new framework to study the matching problem of surfaces equipped with a signal (functional shapes or fshapes). 
The matching energy proposed in \cite{ABN} considers the $L^2$ norm of the signal as a penalty term and a varifold-type attachment term. 
To overcome the difficulties linked to the discretization of weak varifold norm, the authors develop the model in the framework of  Reproducing Kernel Hilbert Spaces (RKHS), which allows them to compute a varifold norm via the dual representation theorem in Hilbert spaces.

As usual, for minimization problems on surfaces, the question of error estimates for discrete solutions arises. The definition of a suitable discrete framework is central to ensuring that discrete solutions give a good approximation of continuous ones. 
In particular, the approximation quality should improve when the triangulation approximating the surface is finer (for example, in the context of finite elements, when the diameter of the triangles of the triangulation is sufficiently small).

 Unfortunately, this is not true in general, and a very famous example called the Schwartz lantern (see \cite{Bronstein} Section 3.9) shows that a cylinder can be approximated  (in the Hausdorff topology) by a sequence of triangulations whose areas diverge.
This example shows that the discretization process of a surface can highly affect the quality of the optimal solution. We finally note that, although we focus on the fshapes matching problem, the previous remarks hold for all variational problems on surfaces.

This paper aims to determine the conditions ensuring that the discrete solution represents a consistent (in terms of geometry) and close (in terms of energy) approximation of the continuous optimal solution.
In particular, we prove that when the discretization process verifies some geometrical conditions, the discrete solution is a good approximation of the continuous one. The proof is given in the framework of the $\Gamma$-convergence theory, proving in particular that the discrete minimizers converge to the continuous one as the diameter of the triangles in the triangulation goes to zero. Error estimates are well established for Euclidean finite elements. However, to our knowledge, this kind of result has not been established for variational problems on surfaces and represents the main contribution of this paper. 
  
\paragraph{Contributions.} 
We call functional shape (fshape)  any couple $(X,f)$ composed of a smooth surface  $X$  with boundary and a signal $f$  defined on $X$.

A functional shape defines a varifold by considering the measure  $\var{X}{f} = \Haus \res X \otimes \delta_{T_X(x)} \otimes  \delta_{f(x)} $.
We study the matching problem between two fixed surfaces $X,Y$ to optimize the signal $f$ on $X$ with respect to a target signal $g$ defined on $Y$. The optimal signal is found by minimizing the following  matching energy
\begin{equation*}%\label{energy-initial}
	E(f) = \|f\|_{BV(X)} + \normvar{\var{X}{f}-\var{Y}{g}}^2
\end{equation*}
where $\|.\|$ denotes a dual norm, and the minimum is taken on $BV(X)$. 
%An existence result for the continuous problem is given in Section \ref{BVfunctional}.

This formulation differs slightly from the original model presented in \cite{ABN}. First, we consider a $BV$ penalty term instead of the $L^2$ norm, proving the result for classic non-regular metrics. The regularity of $f$ can strongly influence the behavior of the optimal solution. Standard $L^2$ signals allow one to work with smooth norms, but accumulation or oscillation phenomena may appear in the optimal solution. For this reason, gradient-dependent norms are increasingly used in image processing to guarantee more realistic optimal configurations. Furthermore, we prove that our results generalize to penalty terms $L^2$ or $H^1$. 

The other main difference concerns the definition of the dual norm. This work considers a standard dual norm instead of the RKHS-based norm proposed in \cite{ABN}. This choice provides the result in a more general context and allows it to be read even without expertise in RKHS theory (Remarks \ref{choiceW} and \ref{l2-model-exist} show how to adapt our results to the RKHS-based norm).
Finally, the proof is given in the case of fixed geometry ($X$ is fixed), and the generalization to the optimization problem concerning $(X,f)$ will be addressed in further work.

The first problem we address concerns the definition of admissible triangulation to obtain a geometrically consistent approximation of surfaces. 
To ensure a complete discretization of the surface in a neighborhood of its boundary, we consider a larger triangulation whose excess part of the projection of X has a small area (Definition \ref{admis-triang}). This allows us to overcome the bijection problems at the boundary of $X$.

Moreover, as explained above, the proximity in terms of Hausdorff distance does not guarantee a consistent surface approximation. The cited example of Schwartz lantern shows that when approximating a cylinder with triangulated surfaces, the area of the triangulations can even diverge depending on the geometric properties of the triangles. In particular, \cite{MorvanThibert} shows that convergence of areas can be ensured by a condition on the angle between the normal vectors at corresponding points (via the projection) of the surface and the triangulation. This additional requirement is added to the definition of admissible triangulations to guarantee a convenient approximation of the related surface (see Lemma \ref{lifting} and Hypothesis \ref{hyp0l2}). 

Once the set of admissible triangulations is defined, we recall the discrete version of the problem following \cite{ABN}. Section \ref{Discretization} describes in detail how to set the problem in the setting of finite elements via the projection of the triangulations on the surface.
Then, the continuous problem can be approximated by a sequence of discrete problems defined on some triangulations whose triangle diameter goes to zero. The triangle diameter is also the parameter indexing the family of discrete problems, which ensures the convergence results for fine triangulations.

The main result proves that discrete solutions approximate continuous minimizers. The result is proven using the  $\Gamma$-convergence theory, which is a notion of convergence of functionals introduced by E. De Giorgi \cite{Ennio, Braides}, allowing to justify the transition from discrete to continuous problems. $\Gamma$-convergence guarantees, in particular, the convergence of the discrete minimizers to the continuous one as the diameter of the triangles goes to zero (see Theorems  \ref{gamma-conv} and \ref{minima}).

We finally point out that, beyond the specificities of the matching problem for functional shapes,  the quality problem for discrete approximations concerns many numerical problems defined on surfaces. This work shows how the approximation quality depends on the discretization process used for the numerical approach. Moreover, $\Gamma$-convergence is the proper framework to establish the consistency of numerical results. Then, our approach can generalize to other problems with promising theoretical and numerical results.

\paragraph{Structure of the paper.} 
In Section \ref{BVmanifolds}, we present the theory of $BV$ functions on manifolds with boundary and adapt the classical results of approximation and compactness. In Section \ref{fvarifold}, we recall the framework of functional varifolds and the link with \cite{ABN}. In Section \ref{BVfunctional}, we set the continuous problem and prove the related existence result. 
Section \ref{Triangulations} is dedicated to the definition of admissible triangulations (Definitions \ref{admis-triang}) and the geometric conditions to guarantee the areas convergence (Hypothesis \ref{hyp0l2} and Proposition \ref{conv-area}). Section \ref{Discretization} defines discrete operators on triangulations in the framework of finite elements.
Finally, in Section \ref{GammaConvSection}, we consider the discrete problem and prove the $\Gamma$-convergence result (see Theorem  \ref{gamma-conv} and \ref{minima}).

\section{$BV$ functions on manifolds}\label{BVmanifolds}

In this section, we introduce the central notions and results about $BV$ functions on manifolds. To develop a much larger setting beyond our framework, a similar definition is given in \cite{BV_manifold}.

Although this work concerns approximation results on surfaces, we present here the $BV$ theory on a slightly more general framework where $X$ is a general compact $d$ dimensional manifold (that is not supposed to be a finite 2D compact submanifold on $\mathbb{R}^3$).

The definition of $BV$ functions depends on introducing a divergence operator (or, equivalently, a volume form) and a local notion of length. A Riemannian structure provides these two things. 

Let $X$ denote an oriented smooth (at least $C^1$) compact $d$-dimensional Riemannian manifold, possibly with boundary denoted $\dX$, and let $\oX$ be the associated Riemannian volume form. The boundary $\dX$ is supposed to be a $C^1$ compact $(d-1)$-dimensional manifold and we have $\oX(\partial X)=0$. Finally, let us denote
\[
X_0= X\setminus \dX
\] 
which is a $C^1$ manifold without boundary. When $X$ is without boundary, the previous construction gives $X_0=X$. We say that $f\in L^1(X)$ is a function of bounded variation on $X$ if
\[
|D_Xf|(X)= \sup\bigg\{\ \int_{X} f\mdiv_X(u)\oX\ |\  u\in \chi^1_c(X_0),\ \|u\|_\infty\leq 1\bigg\}<\infty
\]
where $\chi^1_c(X_0)$ denotes the set of $C^1$ vector fields $u:X \to TX$ on $X$ compactly supported in $X_0$ and $\mdiv_X$ is the divergence operator on $X$ defined by $$\mdiv_X(u)=\sum_{i=1}^d g(e_i,du(e_i))$$ where $(e_1,\cdots,e_d)$ is an orthonormal frame on $TX$. Here $\|u\|_\infty=\sup_{x\in X}g_x(u(x),u(x))^{\frac{1}{2}}$ where $g$ is the metric tensor associated with the Riemannian structure. 

We recall the integration by parts formula :
\begin{equation}
  \label{eq:13}
  \int_X f\mdiv_X(u)\oX=-\int_X u(f)\oX=-\int_Xg(\nabla f,u)\oX
\end{equation}
where $h\in C_c^1(X_0)$ and $u(f)$ denotes the derivative of $f$ along the vector fields $u$ that is defined by $[u(f)](x)=d_xf(u(x))$ for any $x\in X_0$. We retrieve the usual submanifold setting when considering the metric induced on the submanifold by the ambient space $\mathbb{R}^3$. 

The functional space $BV(X)$ endowed with the norm 
$$\|f\|_{BV(X)} = \|f\|_{L^1(X)} + |D_Xf |(X)$$
is a Banach space. 
%Remark that the previous norm is powerful. For example, if $X\subset \R^2$, then  $C^1$-functions are not dense in $BV(X)$ with respect to such a topology. In fact if we consider $f\in BV(X)$ ($X\subset \R^2$) such that $D_Xf$ is not zero and $D_Xf\perp \mathcal{L}^n$ then for every $g\in C^1(X,\RR)$, since $D_Xg<<\mathcal{L}^n$, we have$$|D_X(f-g)|(X) = |D_Xf|(X) + |D_Xg|(X) \geq  |D_Xf|(X)\,.$$

\begin{defn}The space $BV(X)$ can also be equipped with the following notions of convergence, both weaker than the norm convergence:
\begin{enumerate}
			\item {\bf Weak-$\boldsymbol *$ topology}. Let $\{f_h\}_h\subset BV(X)$ and $f\in BV(X)$. We say that the sequence $\{f_h\}_h$ weakly-$*$ converges in $BV(X)$ to $f$ if
				\[f_h \overset{L^1(X)}{\longrightarrow} f \quad\mbox{and}\quad D_X f_h \overset{*}{\rightharpoonup}D_X f\,, \quad\mbox{as}\quad h \rightarrow \infty\]
				%{\color{red} Il faudrait rappeler la definition de $D_X f_h \overset{*}{\rightharpoonup}D_X f$ }
				where $\overset{*}{\rightharpoonup}$ denotes the weak convergence in the space of measures on $X$.
			\item {\bf Strict topology}. Let $\{f_h\}_h\subset BV(X)$ and $f\in BV(X)$. We say that the sequence $\{f_h\}_h$  strictly converges to $f$ in $BV(X)$ if
				\[f_h \overset{L^1(X)}{\longrightarrow} f \quad\mbox{and}\quad |D_X f_h|(X) \longrightarrow |D_X f|(X)\,,\quad\mbox{as}\quad h\rightarrow \infty.
				\]
			Remark that 
			\[d(f,g) = \|f-g\|_{L^1(X)} +\Big||D_Xf|(X)-|D_Xg|(X) \Big|\]
		is a distance in $BV(X)$ inducing the strict convergence.
\end{enumerate}\end{defn}

\noindent We recall the main compactness result for functions of bounded variation extended to the manifold framework.

\begin{thm}[{\bf Compactness}]\label{CompBV}
Let $X$ be a smooth manifold and $\{f_h\}_h$ be a sequence of $BV(X)$ such that $\|f_h\|_{BV(X)}$ is uniformly bounded. Then $\{f_h\}_h$ is relatively compact in $BV(X)$ with respect to the  weakly-$*$ convergence.
\end{thm}

\begin{proof}
The proof consists in adapting that one given for the Euclidean case (see \cite{AFP}, Theorem 3.23) to a manifold with a boundary.

As $X$ is compact, we can consider a finite atlas $\{(U_i,\varphi_i)\}_{i=1,...,n}$ where $\{U_i\}_{i=1,...,n}$ is a finite open cover of $X$ (note that $\{U_i\cap X_0\}_{i=1,...,n}$ is a finite open cover of $X_0$) and  $\varphi_i : U_i \rightarrow V_i$ is at least a $C^1$ diffeomorphism  onto an open $V_i\subset \R^k$ ($k$ is the dimension of the manifold). Without loss of generality, we can suppose that each $V_i$ is a Lispchitz domain, and we can also consider a partition of unity $\{\eta_i\}_{i=1,...,n}$ on $X$ subordinate to the cover $\{U_i\}_{i=1,...,n}$.

Now, by Proposition 2.2 in \cite{BV_manifold}, as $\{f_h\eta_i\}_h \subset BV(U_i\cap X_0)$, then  $\{(f_h\eta_i)\circ \varphi_i^{-1}\}_h \subset BV(\varphi_i(U_i\cap X_0))$ and we can apply the classical theorem  (see \cite{AFP}, Theorem 3.23) to $\{(f_h\eta_i)\circ \varphi_i^{-1}\}_h$ for any local chart. The extension argument used in the classical compactness theorem allows us to define a subsequence converging in $L^1(V_i)$.

Then, via at most $n$ extractions of subsequences, we can define a subsequence (not relabeled) such that, for every $i$,    $\{f_h\eta_i\}_h$ weakly*-converges to some $f_i\in BV(U_i)$, so that $\{f_h\}$ weakly*-converges to $f = \sum_{i=1}^n f_i$ in $BV(X)$. 
\end{proof}

We now establish an approximation result of $BV$ functions by smooth functions. Interestingly, even if there exists several density results using smoothing through the heat kernel semi-group for geodesically complete Riemannian manifold (see \cite{miranda2007heat}, \cite{carbonaro2007note} and \cite{guneysu2013functions}), to the best of our knowledge, an approximation result in the case of a compact manifold with boundary does not seem to be available. We give below such a result~:
\begin{thm}\label{approx1}
  Let $X$ be an orientable Riemannian compact manifold with boundary $\partial X$ and let $X_0=X\setminus \partial X$. For any $f\in BV(X)$ there exists a sequence $\{f_h\}_h\subset C^1(X_0)$ such that 
\begin{equation}
f_h\to \rst{f}{X_0} \text{ in }L^1(X_0,\R)\quad
\text{and}\quad
|D_Xf|(X)=\displaystyle{\lim_{n\to\infty}\int_{X_0}|\nabla f_h| \oX}\,.
\label{eq:14.10.1a}
\end{equation}
\end{thm}

\begin{proof}
The proof shares similar ideas to the proof of the classical approximation result in the Euclidean case %of an open set $\Omega\in \mathbb{R}^d$ 
(see Theorem 3.9 in \cite{AFP}), Furthermore, it is based on the parameterization of the function by the local charts. 
When considering the problem in local charts, introducing a non-constant volume term introduces several new elements. The proof is detailed in the Appendix.
\end{proof}

In the following, we return to the case of a submanifold of $\R^3$ and show how the previous result can be improved. To achieve this goal, we need to introduce an extension operator to be able to expand the manifold and its signal in a consistent way.

Let $X$ denote a $C^p$ ($p\geq 2$) compact oriented 2 dimensional submanifold of $\R^3$  with non empty boundary denoted $\dX$. We denote $\bfn_{X_0}$ the $C^{p-1}$ vectors field of positively oriented normal along $X_0 = X \setminus \partial X$.  Note that $\bfn_{X_0}$ can be continuously extended on $X$ and denoted in that case $\bfn_X$. For any $x\in \partial X$, we can define the unit vector $\nu(x)$ pointing outward and orthogonal to $n_X(x)$ and $\partial X$.  % (i.e. if $(e_1,e_2)$ is a positively oriented orthonormal frame locally defined around a point $x\in X$, $(e_1(x),e_2(x),\bfn(x))$ is an orthonormal frame of $\mathbb{R}^3$). 
For sufficiently small $r>0$, we consider the open subset 
\[
	\Nc_r(X_0)=\{\ x+t\bfn_{X_0}(x)\ |\ x\in X_0, |t|<r\}
\]
and the $C^{p-1}$ diffeomorphism
$$\psi_{X_0}:  X_0\times]-r, r[ \rightarrow \Nc_r(X_0) \;, \quad (x,t)\mapsto x+t\bfn_{X_0}(x)$$
so that, considering its inverse $\psi_{X_0}^{-1}=(\pi_{X_0},t_{X_0})$, we have 
$$ \forall z\in \Nc_r(X_0) \quad z=\pi_{X_0}(z)+t_{X_0}(z)\bfn_{X_0}(\pi_{X_0}(z))$$
where $\pi_{X_0}(z)$ is the orthonormal projection of $z$ on $X_0$ and $|t_{X_0}(z)|$ is the distance of $z$ to $X_0$.

\noindent Now, there exists $\eta_0>0$ small enough, such that the mapping 
$$\Ext:\partial X\times ]-\eta_0,\eta_0[\times ]-r,r[\to \mathbb{R}^3 $$
\begin{equation}\label{ext}
\Ext(x,s,t)=
\begin{cases}
	x_0+t\bfn_{X_0}(x_0)\text{ with }x_0\doteq \pi_{X_0}(x+s\nu(x))&\text{ if }s<0\\
	x+s\nu(x)+t\bfn_X(x)&\text{otherwise} 
\end{cases}
\end{equation}
is well defined and is a $C^{p-1}$ diffeomorphism on an open neighborhood of $\partial X$ in $\mathbb{R}^3$  (see Figure \ref{fig:ext}).
\begin{figure}[h!]
	\centering
	\begin{subfigure}[t]{8cm}
	\centering
		\begin{tikzpicture}[scale=.8]
\begin{axis}[axis lines=none,
	axis equal=true,
view={27}{20},
after end axis/.code={
	\filldraw[opacity=0] (Xx)  node[above, opacity=0] {$\Ext(x,s,t)$} circle (1pt);
		\draw[gray] (tmin) -- (tmax) node[left]{$t$};
		\draw[thick,->] (axis cs:2,0,4/5)(X00)  -- (n)node[left] {$\bfn_{X_0}(x_0)$} ;
		\filldraw (X)  node[right] {$\Ext(x,s,t)$} circle (1pt);
		\draw[thick,->] (X0)  node[left]{$x$} -- (nu)node[below right] {$\nu(x)$} ;
	},rotate=10,
]

 \def\t{.6};
    \def\s{-2};
    \def\tmin{-1};
    \def\tmax{1.5};
    \def\smax{-2.5};
    		\coordinate (X0) at (axis cs:2,0, 4/5 );
    		\coordinate (X00) at (axis cs:0,0, 0 );
    		\coordinate (n) at (axis cs:0,0,1/2 );
		\coordinate (tn) at (axis cs:2-4/5*\t,0, 4/5+1*\t );
		\coordinate (nu) at (axis cs:2+1*1.113552873/2,0,4/5+4/5*1.113552873/2);
		\coordinate (snu) at (axis cs:2+1*\s,0,4/5+4/5*\s);
		\coordinate (X) at (axis cs:0,0,1.2);
		\coordinate (smax) at (axis cs:2+1*\smax,0,4/5+4/5*\smax);
		\coordinate (tmin) at (axis cs:0,0,\tmin );
		\coordinate (tmax) at (axis cs:0,0,\tmax);

		\filldraw[gray] (snu) circle (1pt);
		\draw[thick,dashed] (snu) node[right]{$x+s\nu(x)$} -- (X00)node[left] {$x_0$} ;
		\draw[gray] (X0) -- (smax)  node[left]{$s$};
		\filldraw (X00) circle (1pt);
		%\draw[thick,->] (X0) node[left]{$x$} -- (axis cs:2+1/2,0,4/5+4/10)node[right] {$\nu(x)$} ;
    \addplot3[patch,patch refines=4,
		mesh,
		%shader=faceted interp,
		patch type=biquadratic,
	    fill opacity=.25,
	    opacity=.25,
	] 
    table[z expr=.2*x^2-.2*y^2]
    {
        x  y
        -2 -2
        1  -2
        1  2
        -2 2
        0  -2
        2  0
        0  2
        -2 0
        0  0
    };

 \def\tp{.6};
    \def\sp{.75};
		\coordinate (Xx) at (axis cs:2-4/5*\tp+1*\sp,0, 4/5+1*\tp +4/5*\sp);
   \end{axis}

\end{tikzpicture}
		\caption{$s<0$ }
	\end{subfigure}
	\begin{subfigure}[t]{8cm}
	\centering
		\begin{tikzpicture}[scale=.8]
\begin{axis}[axis lines=none,
	axis equal=true,
view={27}{20},
after end axis/.code={
		\draw[dashed] (snu) -- (X) node[above] {$\Ext(x,s,t)$};
		\filldraw (X) circle (1pt);
		\draw[dashed] (tn) --  (X);
		\draw[gray] (tmin) -- (tmax) node[left]{$t$};
		\draw[gray] (X0) -- (smax)  node[right]{$s$};
   		\draw[thick,->] (X0) node[left]{$x$}  -- (n)node[below left] {$\bfn_X(x)$} ;
		\draw[thick,->] (X0) -- (nu)node[below right] {$\nu(x)$} ;
	},
,rotate=10,]
    \addplot3[patch,patch refines=4,
	    mesh,
		patch type=biquadratic] 
    table[z expr=.2*x^2-.2*y^2]
    {
        x  y
        -2 -2
        1  -2
        1  2
        -2 2
        0  -2
        2  0
        0  2
        -2 0
        0  0
    };
 \def\t{.6};
    \def\s{.75};
    \def\tmin{-1};
    \def\tmax{1};
    \def\smax{1};
    		\coordinate (X0) at (axis cs:2,0, 4/5 );
    		\coordinate (n) at (axis cs:2-4/5*1.113552873/2,0, 4/5+1*1.113552873/2 );
		\coordinate (tn) at (axis cs:2-4/5*\t,0, 4/5+1*\t );
		\coordinate (nu) at (axis cs:2+1*1.113552873/2,0,4/5+4/5*1.113552873/2);
		\coordinate (snu) at (axis cs:2+1*\s,0,4/5+4/5*\s);
		\coordinate (X) at (axis cs:2-4/5*\t+1*\s,0, 4/5+1*\t +4/5*\s);
		\coordinate (smax) at (axis cs:2+1*\smax,0,4/5+4/5*\smax);
		\coordinate (tmin) at (axis cs:2-4/5*\tmin,0, 4/5+1*\tmin );
		\coordinate (tmax) at (axis cs:2-4/5*\tmax,0, 4/5+1*\tmax );

   \end{axis}

\end{tikzpicture}
		\caption{$s\geq 0$}
	\end{subfigure}
		\caption{The map $\Ext$}
	\label{fig:ext}
\end{figure}

\noindent Moreover, $\Ext$ maps $\partial X\times \{0\}\times\{0\}$ to the boundary $\partial X$ of $X$ and for $0<\eta<\eta_0$
$$X^{-\eta}= X\setminus \Ext(\partial X\times ]-\eta,0]\times \{0\})$$
is a compact $C^{p-1}$ submanifold of $X$ whereas
\begin{equation}\label{extension-eta}
X^{+\eta}=X\cup \Ext(\partial X\times ]0,+\eta]\times \{0\})
\end{equation}
is a $C^{p-1}$ compact 2 dimensional submanifold of $\mathbb{R}^3$ extending $X$ along its boundary.%  In particular, if 
% $$\Delta X^\eta\doteq X_0^{-\eta}\cup \partial X\cup X_0^{+\eta}$$
% then $\Delta X^\eta$ is a $C^{p-1}$ open submanifold containing the boundary $\partial X$.

We can now prove the following approximation result:

\begin{thm}\label{approx2}
 Let $X$ be a $C^p$ ($p\geq 2$) compact oriented 2D submanifold of $\R^3$  with non empty boundary denoted $\dX$.  Let $f\in BV(X)$ and let $\epsilon>0$. Then there exists $\eta>0$ and $\tilde{f}\in C^{p-1}(\xeta)$ such that 
\[
\int_{X}|f-\tilde{f}|\,d\cH^2+\left||D_Xf|(X)-|D_X\tilde{f}|(X)\right|\leq \epsilon
\]
and 
\[
	\int_{\xeta\setminus X^{-\eta}}|\tilde{f}|+|\nabla_{\xeta }\tilde{f}|\,d\cH^2\leq \epsilon\,.
\]
\end{thm}

\begin{proof}
By Theorem \ref{approx1}, there exists $f'\in C^{p-1}(X_0)$ such that
\[
\int_{X}|f-f'|\,d\cH^2+\Big| |D_Xf|(X)-|D_Xf'|(X)\Big|\leq \epsilon\,.
\]
Moreover, there exists $0<\eta'<\eta_0$ such that 
$$\int_{X\setminus X^{-\eta'}}|f'|+|\nabla_X f'|\,d\cH^2\leq \epsilon$$
Now, considering for $\eta=\eta'/3$ the function $f'':\xeta\to\mathbb{R}$ defined as
\[ f''( z)=
  \begin{cases}
f'(z) & \text{ if }z\in X^{-\eta}\,,\\
f'(\Ext(x,-2\eta -s,0)) &\text{ if }z\in \xeta \setminus X^{-\eta}  \text{ and where } z=\Ext(x,s,0)\,.
\end{cases}
\]
Let $\delta>0$, we can easily check that for $\eta'$ small enough
\[
\int_{\xeta\setminus X^{-\eta}}|f''|+|\nabla_X f''|\,d\cH^2\leq (1+\delta)\times 2\int_{X\setminus X^{-\eta'}}|f'|+|\nabla_X f'|\,d\cH^2
\]
since by considering the change of variable induced by $\Ext$ we notice that $(x,s)\to \pi_X(x+s\nu(x))$ has a determinant converging to $1$ when $(x,s)$ convergence to a point on  $\partial X\times\{0\}$ with $s<0$. 

Moreover, we verify that $\rst{f''}{X^{-\eta}}\in C^{p-1}(X^{-\eta})$ and $\rst{f''}{\overline{\xeta\setminus X^{-\eta}}}\in C^{p-1}(\overline{\xeta\setminus X^{-\eta}})$ where the intersection $X^{-\eta}\cap \overline{\xeta\setminus X^{-\eta}}=\partial X^{-\eta}$ is a $C^{p-1}$ one dimensional submanifold. Applying a smoothing in the vicinity of $\partial X^{-\eta}$ we can obtain $\tilde{f}\in C^1(\xeta,\mathbb{R})$ such that  $\tilde{f}=f''$ on $ X^{-\eta} $ and 

$$\int_{\xeta\setminus X^{-\eta}}|\tilde{f}|+|\nabla_X \tilde{f}|\,d\cH^2\leq \int_{\xeta\setminus X^{-\eta}}|f''|+|\nabla_X f''|\,d\cH^2+\epsilon$$
so that, choosing $\delta=1$ and $\eta'$ small enough, we have
$$\int_{\xeta\setminus X^{-\eta}}|\tilde{f}|+|\nabla_X \tilde{f}|\,d\cH^2\leq 5\epsilon\,,$$
and we get immediately
$$\int_{X}|f-\tilde{f}|\,d\cH^2+\left||D_Xf|(X)-|D_X\tilde{f}|(X)\right|\leq 
\int_{X}|f-f'|\,d\cH^2+\Big||D_Xf|(X)-|D_Xf'|(X)\Big|+5\epsilon\leq 6\epsilon$$
that proves the result.
\end{proof}

\begin{rem}[{\bf $L^2$ and $H^1$ norms}]\label{rem-approx-L2} Theorems \ref{approx1} and \ref{approx2} can be established, by similar arguments, for the $L^2$ and  $H^1$ norms.  We refer to \cite{Hebey} for an introduction to Sobolev spaces on manifolds.
\end{rem}

\section{Functional varifolds}\label{fvarifold}

This section presents the main definitions and results for functional varifolds as introduced in \cite{ABN}. We recall that this framework is a generalization of the varifold theory \cite{Si} to spaces with a scalar component; this allows the representation of geometric shapes equipped with a signal in the framework of measure theory. We refer to  \cite{Si, ABN} for a more detailed presentation of the respective theories.  

The functional approach considers mathematical objects, called fshapes,  containing geometrical and functional components. 
As detailed below, the role of fshapes in the functional varifolds framework is the same as that of manifolds in the varifold theory; they allow one to define measures supported on surfaces and are the bridge between geometry and measure theory.

In the following, we work with smooth surfaces and denote by $X$ a generic $2$-manifold verifying the following properties :

\begin{hyp}\label{cond_surf}
$X$ denotes a $2$-submanifold of $\R^3$ (surface) with boundary $\partial X$ and which is smooth ($C^2$  at least), oriented, connected and compact.
The smoothness assumption implies that its interior, denoted by $X_0$, is a two-dimensional smooth manifold, and $\partial X$ is a one-dimensional manifold of the same regularity.
\end{hyp}

\begin{defn} [{\bf fshape}] We define a fshape as a couple $(X,f)$ where $X$ is a surface verifying Hypothesis \ref{cond_surf} and $f\in BV(X)$ is a signal defined on $X$.  
\end{defn}

Of course,  the penalty term in the matching problem defines the regularity of the signal. In the following, we consider fshapes endowed with $L^2$ or $H^1$ signal when studying the problem with $L^2$ or $H^1$  penalty terms, respectively.

Similarly to the classical theory of varifolds,  we can define a functional varfiold via the dual of $\test$, which denotes the closure, with respect to the $C^1$ norm, of the set of $C^1$ functions with compact support on $\R^3\times G(3,2)\times \R$. We denote by $G(3,2)$ the Grassmannian of the non-oriented 2-dimensional linear subspaces of $\R^3$.

\begin{defn} [{\bf fvarifolds}] A 2-dimensional functional varifold (fvarifold) is any operator $\mu$
  belonging to $(\test)'$, the dual space  of $\test$, containing all bounded linear real-valued map on $\test$.  In the following, we consider the following norm for fvarifolds :
\begin{equation}\label{dual-emb}
\normvar{\mu} = \sup\,\{\mu(\varphi)\,:\,\|\varphi\|_{\test}\leq 1\}\,.
\end{equation}

\end{defn}
We can associate a fvarifold to a given fshape $(X,f)$ by considering the following measure
$$\var{X}{f} = \Haus \res X \otimes \delta_{T_X(x)} \otimes  \delta_{f(x)} $$
that acts as a linear functional in the following way :
\begin{equation}\label{action-rect}
\var{X}{f}(\varphi) = \int_X\varphi(x, T_x X, f(x)) d \Haus(x)\quad \forall \varphi\in \test\,,
\end{equation}
where $T_X(x)$ denotes the tangent space to $X$ at $x$ and $\Haus$ is the $2$-dimensional Hausdorff (or volume) measure.

Finally, we prove some lemmas that will be useful in the following, concerning some convergence properties for fvarifolds supported on fhsapes.
In particular, the convergence to the null fvarifold depends on the geometric component :
\begin{lem}\label{null-var}
Let $\{X_h\}$ be a sequence of surfaces such that $\Haus(X_h)\rightarrow 0$. Then $\var{X_h}{f_h}$ converges to the null fvarifold for every sequence $\{f_h\}$ of signals.
\end{lem}

\begin{proof} For every $h$ and for every $\test$ with $ \|\varphi\|_{C_0(\R^3\times G(3,2)\times \R)}\leq 1$ we obtain that $|\var{X_h}{f_h}(\varphi)|\leq \mathcal{H}^2(X_h)$ which implies that $\|\var{X_h}{f_h}\|$ converges to zero.
\end{proof}

The following lemma links the convergence of signals and the weak-$*$ convergence of fvarifolds :
%It follows from the definition of varifold norm and the continuity of test functions applied to \eqref{action-rect} : 

\begin{lem}\label{conv-var}
If $f_h \rightarrow f$ in $L^1(X)$ then 
$\var{X}{f_h}\overset{*}{\rightharpoonup} \var{X}{f}$
and   $\normvar{\var{X}{f_h}- \var{X}{f}}\rightarrow 0$.
\end{lem}

\begin{proof} For every $\varphi \in \test$  with $ \|\varphi\|_{\test}\leq 1$ we get 
$$ |\var{X}{f_h}(\varphi) - \var{X}{f}(\varphi) | \leq \int_X|\varphi(x, T_x X, f_h(x)) - \varphi(x, T_x X, f(x))| d \Haus(x) \leq \|f_h-f\|_{L^1(X)}$$ 
which proves the results by taking the supremum and using the $L^1$ convergence. 
\end{proof}

\begin{rem}[{\bf The choice of the fvarifold norm}]\label{choiceW}
This work uses the dual norm as a distance between fvarifolds. However, in \cite{ABN}, another metric for fvarifolds is defined by using the framework of  Reproducing Kernel Hilbert Spaces (RKHS) \cite{Younes}. We refer to \cite{ABN} for a general definition, and we recall that their framework allows them to define the following dual product for fvarifolds supported on fshapes:

%_{W'}
\begin{equation}\label{var-prod}\langle \mu_{(X,f)}, \mu_{(Y,g)} \rangle_{W'} = \int_X \int_Y \,k_e(x,y)k_t(T_x X, T_y Y)k_f(f(x), g(y)) \, d \Haus(x) d \Haus(y)\,,
\end{equation}
with
\begin{equation}\label{kernels}
	k_e(x,y)= e^{-\frac{\|x_1-x_2\|^2}{\sigma_e^2}}\,,\; k_t(T_1, T_2)= e^{-\frac{2(1-\langle \bfn_{T_1}, \bfn_{T_2} \rangle^2)}{\sigma_t^2}}\,,\;k_f(a,b)= e^{-\frac{|f_1-f_2|^2}{\sigma_f^2}}
\end{equation}
where $\sigma_e, \sigma_t, \sigma_f$ are three positive constants and $\bfn_{T}$ represents the unit normal vector to $T$. In particular, $W'$ denotes the dual space of $W$, the RKHS associated with the kernel $k_e\otimes k_t\otimes k_f$ (see Propositions 2 and 4 in \cite{ABN}). As $W$ is continuously embedded into $\test$, its dual norm defines a metric for fvarifolds.

The choice of this metric has many advantages. It allows them to prove an existence result for the matching problem with $L^2$ signals (see Section \ref{BVfunctional}). Moreover, in numerical applications, using the Gaussian kernel allows the localization of the matching at a given scale.  

In this work, we decided to consider the matching problem defined for $BV$ or $H^1$ signals and use the standard dual norm for fvarifolds. This framework is strong enough to prove the existence of optimal solutions and a $\Gamma$-convergence result. 
%This also makes the work more accessible to readers not used to RKHS-theory. 
Throughout the paper, we will demonstrate how to adapt the results to the original framework introduced in \cite{ABN}.

\end{rem}

\section{Existence of optimal solutions for the matching problem}\label{BVfunctional}

In this section, we define the matching energy between two fshapes and prove an existence result for the optimal solution.

Let be $X$ a surface and $(Y,g)$ a target fshape. We consider the following energy defined for a generic fshape $(X,f)$
\begin{equation}
	E(f) = \|f\|_{BV(X)} + \frac 1 2 \normvar{\mu_{(X,f)}-\mu_{(Y,g)}}^2 \,,
\end{equation}
and we aim to solve  the minimization problem
\[
	\inf_{f\in BV(X)} E(f).
\]
%The $BV$ norm, used as a penalty term, defines a regularity constraint for the matching problem. Moreover, we show how to adapt our results to  $L^2(X)$ and  $H^1(X)$ norms. 

We recall that we optimize only with respect to the signal, which implies that the initial and optimal configurations have the same geometric support. However, the geometry is taken into account in the attachment term. 

\begin{thm}\label{existence} Let $(Y,g)$ a given fshape and $X$ a $2$-manifold verifying  Hypothesis \ref{cond_surf}. Then, there exists at least one solution to the problem
\begin{equation}\label{P}
\underset{f\in BV(X)}{\inf}\; E(f)\,.
\end{equation}
\end{thm}

\begin{proof} Let $\{f_h\}_h$ be a minimizing sequence belonging to $BV(X)$. We can suppose that $\|f_h\|_{BV(X)}$ is uniformly bounded and, by Theorem \ref{CompBV}, $\{f_h\}_h$ converges (up to a subsequence) to some $f\in BV(X)$ with respect to weak-$*$ topology. 
The result follows by remarking that the fvarifold norm is continuous with respect to the $L^1$ topology (Lemma \ref{conv-var}) and that the $BV$ norm is lower semicontinuous with respect to the $L^1$ topology.
\end{proof}

\begin{rem}[{\bf The $H^1$ model}]\label{sob1} We get the same result if we consider signals belonging to $H^1(X)$ instead of $BV(X)$.  It follows from the fact that the unity ball of $H^1(X)$ is compact with respect to the weak topology. Moreover,  $H^1(X)$ is compactly embedded in $L^2(X)$, which implies (up to a subsequence) the $L^1$ convergence of the minimizing sequence. 
\end{rem}
%%% Local Variables:
%%% TeX-master: "ShapeMatchingBV"
%%% End:

%%%%%%%%%%%%%%%%%%%%%%%% L2 %%%%%%%%%%%%%%%%%%%%%%%%%%%

\begin{rem}[{\bf $L^2$ model in \cite{ABN}}]\label{l2-model-exist}

The $L^2$ model is defined (see \cite{ABN}, part I) via the following matching energy 
\begin{equation}\label{Fl2}
E(f) = \frac{\gamma_f}{2} \|f\|_{L^2(X,\R)}^2 + \frac{\gamma_W}{2} \normvarW{\var{X}{f}-\var{Y}{g}}^2
\end{equation}
where  $\gamma_f, \gamma_W$ are two positive constants, and the fvarivold norm is induced by \eqref{var-prod}. 
We point out that, for every minimizing sequence $\{f_h\}_h$, a bound on $E(f_h)$  guarantees only the $L^2$ weak compactness for the signals, which is not enough to get the semicontinuity of the fvarifold term (a result similar to Lemma \ref{conv-var} holds for the $W'$-norm with respect to the $a.e.$ convergence of the signals). This justifies the choice of the RKHS-based metric described in Remark \ref{choiceW}, whose properties play a central role in the existence result :

\begin{thm}[Proposition 7 in \cite{ABN}]\label{l2-ex}

Let $X,Y$ be two finite volume bounded 2-rectifiable subsets of $\R^3$. Let us assume that  $W$ is continuously embedded in $C^2_0(\R^3\times G(3,2) \times \R)$ and $g\in L^2(Y)$.

If the ratio $\gamma_f/\gamma_W$ is large enough, 
then  there exists  at least one solution to the minimization problem
\[\inf_{f\in L^2(X)} E(f)\]
and every minimizer  belongs to $L^\infty(X)$. Moreover, if $X$ is a $C^p$ surface and $W\hookrightarrow C_0^m(\R^3\times G(3,2) \times \R)$ with $m\geq \max\{p,2\}$, then every minimizer belongs to $C^{p-1}(X)$.

Finally, there exists a constant $C>0$ (independent of $X$ and $Y$) such that,  every  minimizer verifies
\begin{equation}\label{bound-min}
\|f\|_{L^\infty(X)} \leq C\frac{\gamma_W}{\gamma_f}(\Haus(X)+ \Haus(Y))\,.
\end{equation}
%(\cite{ABN}: Proposition 6, p. 21)
\end{thm}

The proof is based on the relaxation of the energy to the class $\mathcal{M}^X$ defined as follows:
\begin{defn}\label{MX} $\mathcal{M}^X$ is the class  of the Borel finite measures $\nu$  on $\R^3\times G(3,2) \times \R$ such that 
\[\int \varphi(x,V) d \nu(x,V,f) = \int_X \varphi(x,T_x X) d \Haus(x) \quad \forall \, \varphi \in C_c(\R^3\times G(3,2))\,.\]
Note that $\var{X}{f}\in \mathcal{M}^X$ for every fshape $(X,f)$.
\end{defn}

Then, the energy $E$ can be relaxed to the following functional
\begin{equation}\label{functional-measure}
 \tilde{E}:\mathcal{M}^X\rightarrow \R\;,\quad \tilde{E}(\nu) =  \frac{\gamma_f}{2} \int |f|^2 d \nu + \frac{\gamma_W}{2} \normvar{\nu-\var{Y}{g}}^2
\,
\end{equation}
and it holds : 
\begin{equation}\label{en-en}
E(f) = \tilde{E}(\var{X}{f})\,.
\end{equation}

The relaxed energy provides compactness for minimizing sequences in the space of measures. It can be shown that the minimizing measure $\nu^*$ of $\tilde{E}$ is associated with a fshape, so that $\nu^* = \var{X}{f^*}$ for some $f^*\in L^2(X)$. 
The proof relies on the Implicit Function Theorem, which needs the hypothesis on $\gamma_f/\gamma_W$. We refer to  Proposition 7 and Lemma 2 in \cite{ABN} for more details. 

However, the $L^2$ model has two main issues. Firstly, the existence result depends on the weights used to define the energy. Moreover, the $L^2$ penalty does not prevent some oscillating configurations for the optimal signal.
For this reason, we introduced a penalty on the signal derivative that justifies the $BV$ and $H^1$ models studied in this work. 

\end{rem}

\section{Surfaces and triangulations}\label{Triangulations}

This section is dedicated to defining suitable triangulations of surfaces to ensure a good approximation of continuous functions by their discretization. 

As explained in the Introduction, the quality of approximation and the error estimates highly depend on the geometric properties of triangulations.
The famous example of the Schwartz lantern (see \cite{Bronstein} Section 3.9) exhibits a sequence of polyhedral surfaces converging to the cylinder in the Hausdorff metric and whose areas diverge. This example points out that triangulations must verify some specific hypothesis to represent a geometrically consistent surface approximation. 

Section \ref{tri-adm} recalls some general facts about triangulations and defines the set of admissible triangulations in Definition \ref{admis-triang}. In Section \ref{conv-areas}, we give a sufficient condition (Hypothesis \ref{hyp0l2}), ensuring the convergence of areas for admissible triangulations. Finally, in Section \ref{signal-admi}, we define the sampling method for signals and the error estimate for Sobolev norms.

\subsection{Triangulations of a surface}\label{tri-adm}

%We start by stating some general notions about triangulations related to the surface $X$. 
As $X$ verifies Hypothesis \ref{cond_surf}, following \cite{Alexandrov,Resh,Polthier}, we give the following definition of triangulation: 

\begin{defn}[{\bf Triangulations}]\label{def.tri}
A triangulation $\tri$ is a two-dimensional manifold
(with boundary) consisting of a finite  set $\Delta_\tri$ of affine
triangles  such that:
\begin{enumerate}
\item  any point $p \in\tri$ lies in at least one triangle $T\in \Delta_\tri$;
\item each point $p \in\tri$ has a neighbourhood that intersects only
finitely many triangles of $\Delta_\tri$;
\item  the intersection of any two non-identical triangles $T,T'\in \Delta_\tri$ is
either empty or consists of a common vertex or edge.
\end{enumerate}
In the following, we denote by $T$  a  generic triangle of $\Delta_\tri$ and by $\partial \tri$ the boundary of the manifold $\tri$. The vertices of $\tri$ are the vertices of the triangles belonging to  $\Delta_\tri$. The diameter of $\tri$ is defined as follows :
\begin{equation}\label{eq.diam}
	\diam_\tri = \max_{T \in \Delta_\tri} \{\diam(T)\}.
\end{equation}
We also assume that every triangulation is regular, which means that
\begin{equation}\label{regular-tri}
	\frac{h_T}{\rho_T}\leq C \quad \quad \forall\, T\in \Delta_\tri
\end{equation}
for some $C>0$, where $h_T$ is the diameter of $T$ and  $\rho_T$ is the  diameter of the sphere
inscribed in $T$.
\end{defn}

We define the distance function to $X$ as the following function :
\[
\forall x\in \R^3\;,\quad d_X(x) =  d(x,X) = \inf_{y\in X}\, |x-y|\,.\] 
For every $x\in \R^3$, we call  (if it exists)  projection of $x$ on $X$ every point $\pi_X(x)\in X$ such that  $d_X(x)=  |x-\pi_X(x)|$.  

We also recall that the Hausdorff distance between  two surfaces $X,Y\subset \R^3$ is defined as 
\[
%d_{\mathcal H}(X,Y)=\max \Big\{ \sup_{x\in X} \inf_{y\in Y} \|x-y\|, \sup_{y\in Y} \inf_{x\in X} \|x-y\|  \Big\}
d_{\mathcal H}(X,Y)=\max \Big\{ \sup_{x\in X} d_Y(x), \sup_{y\in Y} d_X(y)  \Big\}.
\]

\begin{defn}[\bf Tubular neighborhoods] 
Let $X$ be a surface satisfying Hypothesis \ref{cond_surf}. We denote by $U_r(X)$ the subset of $\R^3$ of the form
	\[U_r=\,\{\;x\in \R^3\;|\; d_X(x) < r\;\}\,\]
such that every point  $x\in U_r$ admits a unique projection $\pi_X(x)\in X$. We refer to \cite{Federer} for the proof of the existence of such a tubular neighborhood $U_r$ for some $r>0$.

To guarantee the injectivity of the projection, we introduce the following tubular neighborhood
 \begin{equation}\label{eq.N}
 \Nc_r(X) \,=\,\{x+t\normM(x)\,:\,t\in ]-r,r[,\,x\in X \} \subset U_r\,  
 \end{equation}
where we denote by  $\normM(x)$ the unit normal vector to $X$ at $x$.
Then, we have $x\in \Nc_r(X)$ if and only if  $d_X(x)<r$ and
\begin{equation}\label{proj_M}
x=\pi_X(x)+ d_X(x) \normM(\pi_X(x))\,.
\end{equation}

Then, we can split every triangulation into two parts to isolate the set ($\tin$) of points that can be bijectively projected onto the surface:
\begin{equation*}%\label{notations}
\tin = \tri \cap  \Nc_r(X)\quad  \text{ and } \quad \tout = \tri \cap\Nc_r(X)^c\,.
\end{equation*}
\end{defn}

Generally, a triangulation is not in bijection with a surface with a smooth boundary through the normal projection. The bijectivity can fail close to $\partial X$ because of the curvature of $\partial X$ as depicted in Figure \ref{tri_inscribed}. Locally, the normal projection of $\partial X$ on a hyperplane must not be a line, and it is impossible to project $\partial X$ on the edge of a triangle.

Then, we introduce a new class of admissible triangulations to ensure the projection bijectivity between triangulation and surface except on a small part close to the boundary:

\begin{defn}[{\bf $h$-admissible triangulations for a surface}]\label{admis-triang}
Let $h>0$. We say that a triangulation $\tri$ is $h$-admissible for the surface $X$ if the following properties hold
\begin{enumerate}[label=$(\roman*)$]
%\item $\tri\in \Treg$;
%\item $\tri$ lies in $U_h(X)$; \label{hyp.Uh}
\item $\tri$ lies in $\Nc_h(X^{+\eta})$ for some $\eta>0$, where $X^{+\eta}$ is an extension of $X$ defined in \eqref{extension-eta}; \label{hyp.ext}
\item $\tin  = \tri \cap  \Nc_h(X)$ and $X$ are in one-to-one  correspondence through $\pi_X$; \label{hyp.bij}
\item $\Haus(\tout) =O(h)$;\label{hyp.out}
\item $\diam_{\tri}  = O(h)$, where $\diam_{\tri} $  is defined in \eqref{eq.diam}.\label{hyp.diam}
\end{enumerate}
\end{defn}

\begin{figure}[!h]
\centering
\begin{subfigure}[t]{8cm}
\centering
	\includegraphics[width=8cm]{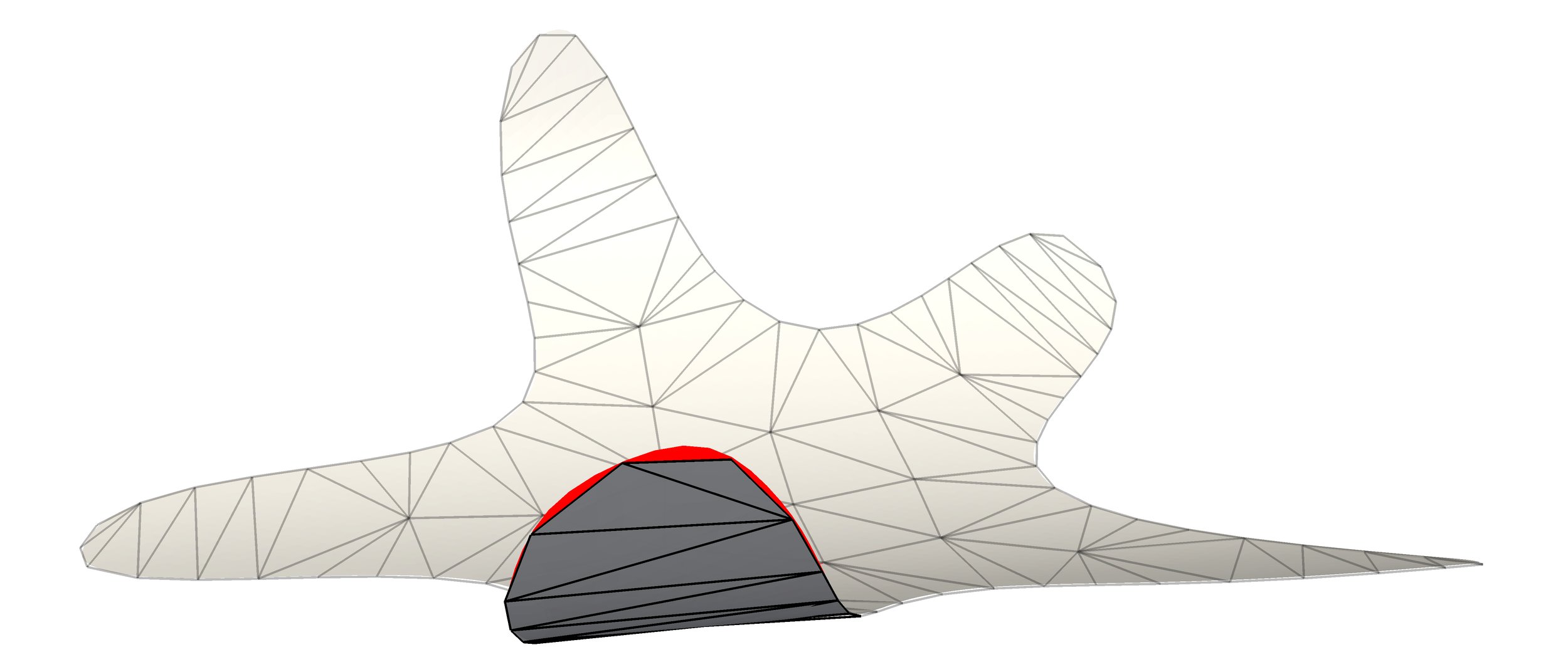}
	\caption{ \small{$X$ can not be entirely projected on $\tri_1$ } } \label{tri_inscribed}
\end{subfigure}
\begin{subfigure}[t]{8cm}
\centering
	\includegraphics[width=8cm]{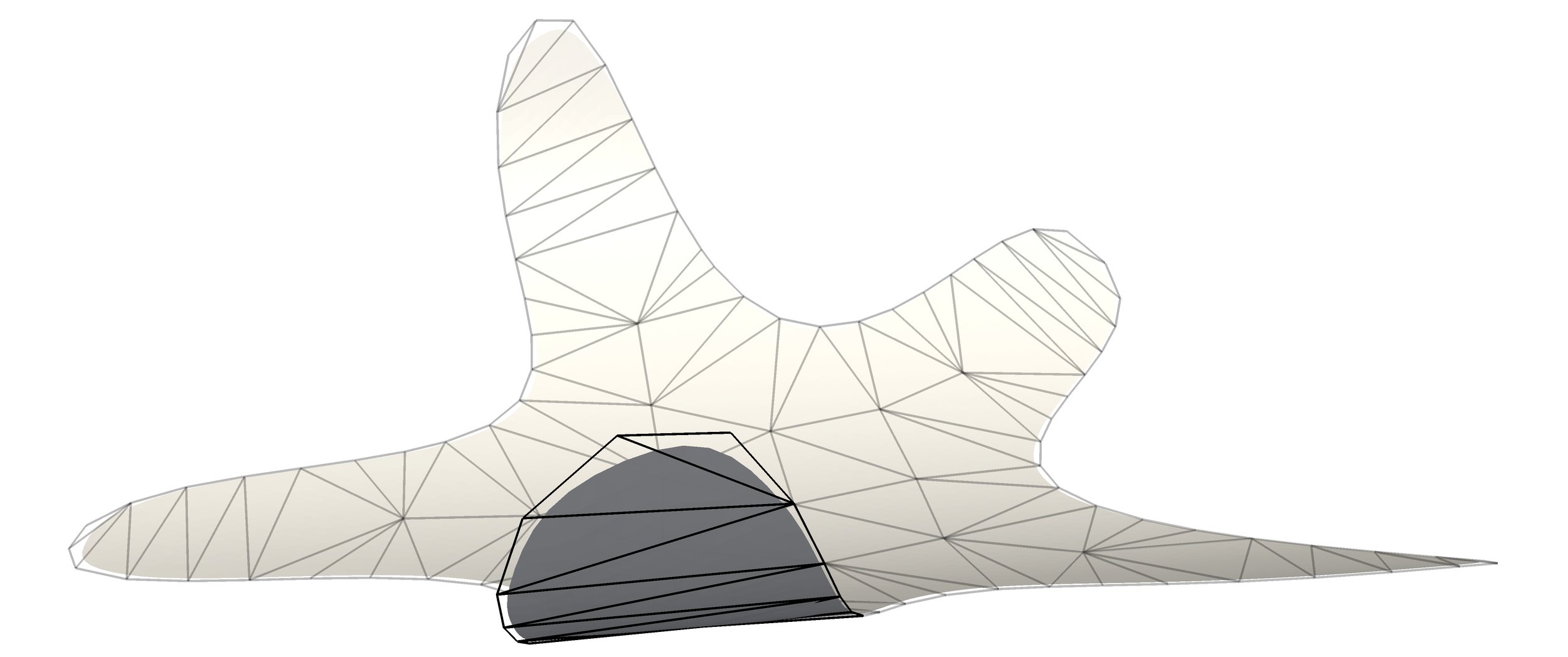}
	\caption{\small{An $h$-admissible triangulation $\tri_2$ for $X$}}\label{tri_hadmissible}

\end{subfigure}
\caption{\small{The surface $X$ is a bent smooth star (solid grey), and two triangulations (black lines) are illustrated. Figure \ref{tri_inscribed}: the triangulation $\tri_1$ is not in one-to-one correspondence with $X$ through the projection map (for instance, the part of the smooth star in red exceeds the triangulation). Figure \ref{tri_hadmissible}: the subset $\tin_2$ of $\tri_2$ is in one-to-one correspondance with $X$}}.
\label{fig-star}
\end{figure}

This definition introduces a discretization framework based on triangulations larger than the corresponding surface but converging to it with respect to the Hausdroff distance (as $h\rightarrow 0$). Condition \ref{hyp.ext} means that $\tri$ is adapted to discretize an extended surface and, because of   \ref{hyp.bij},  $X$ can be completely projected onto the triangulation. However, the part $\tout$, corresponding to the triangles along the boundary,  is assumed to be small via condition \ref{hyp.out}. Finally, condition \ref{hyp.diam} allows us to identify the family parameter with the triangulation diameter in order to describe finer triangulations as $h\rightarrow 0$. Fig. \ref{fig-star} shows an example of admissible and inadmissible triangulations.

We point out that the vertices of triangulations are not, as usual, a set of sampled points on the surface. This approach corresponds better to the data acquisition routines used in image processing.
For instance, biomedical images (OCT, functional MRI) are often modified  (segmentation, deblurring, denoising) to improve their quality for numerical experiments. Then, the data do not correspond to a sampled version of the imaged objects, but instead, they represent an approximation, often noisy, of the natural structures.

\subsection{Convergence of areas}\label{conv-areas}

As pointed out in the previous section, if $\tri$ is an $h$-admissible triangulation of $X$ for some $h>0$ then $ d_\Hh(X,\tri) = O(h)$. However, this does not guarantee a similar error estimate for the respective areas. 

In this section, we introduce some geometric conditions to ensure geometrically consistent triangulations and avoid pathological cases as the already cited  Schwartz lantern (see \cite{Bronstein} Section 3.9).

\begin{defn}[{\bf Angle between normals}]\label{angle} 	Let $h>0$. Assume that $\tri$ is an  $h$-admissible triangulation for $X$.  For every $x\in \tin$ we define the angle $\alpha_x $ as follows
\begin{itemize}
\item  if $x$ belongs to the interior of some triangle, then $\alpha_x$ is the angle belonging to $[0,\pi/2]$ between the two normals ${\bfn}_X(\pi_X(x))$ and ${\bfn}_\tri(x)$;

\item if $x$ belongs to an edge of a triangle, then $\alpha_x$ is the biggest angle belonging to $[0,\pi/2]$ between ${\bfn}_X(\pi_X(x))$ and the normals of the triangles containing $x$.
\end{itemize}
In the following, we set 
\[\alpha_{\max} =  \underset{x\in \tin}{\sup}\, \alpha_x\,.\]
\end{defn}

\begin{lem}\label{lifting}
	Let $h>0$ and $\tri$ be an $h$-admissible triangulation for $X$. We have
\[
	|\Hh^2(X)-\Hh^2(\tin)| = O(\alpha_{\max}^2 + d_\Hh(X,\tri))\,
	%|\Hh^2(X)-\Hh^2(\tin)| = O(\alpha_{\max}^2 + {\color{red} h})\,
\]
as $\alpha_{\max}^2, d_\Hh(X,\tri) \rightarrow 0$,
%$\alpha_{\max}^2,{\color{red} h \to 0}$,
where $\alpha_{\max}$ is introduced in Definition \ref{angle}.
%In particular,  the jacobian of the parameterization \eqref{tri-surf} verifies$$det(Dm) = 1+O(\alpha_{\max}^2 + d_\Hh(X,\tri))\,$$as $\alpha_{\max}^2, d_\Hh(X,\tri)\rightarrow 0$.
\end{lem}

\begin{proof}
For every $x\in X$ we consider on the tangent space $T_xX$ the basis $\Bb(x)=\{e^1(x), e^2(x)\}$ given by the two principal directions. We denote by  $\kappa_1(x)$ and $\kappa_{2}(x)$  the principal curvatures of $X$ at $x$. Similarly, for every $x\in X$ we can consider the basis $\tilde{\Bb}(x)=\{e^1(x), e^2(x), \normM(x)\}$ for $\R^3$.

Consider the differential $D\pi_X: \R^3\rightarrow T_{\pi_X(x)}X $ of  $\pi_X$. Note that, for every $x\in \mathcal{N}_h(X)$, we have  $D\pi_X(x)(v)=0$ for every variation $v$ in the direction ${\bfn}_{X}(\pi_X(x))$ orthogonal to $ T_{\pi_X(x)} X$. Then, we should consider the tangential variations to calculate the projection's Jacobian. 

In \cite{Boris} it is proved that for any $x\in U_h(X)$ and $v$ parallel to $T_{\pi_X(x)}X$ we have
 \begin{equation*}%\label{jacobian}
 D\pi_X(x)(v) = \left(\rst{I}{T_{\pi_X(x)}X}  - \varepsilon_x d_X(x) D\normM(\pi_X(x))\right )^{-1}v
\end{equation*}
where the matrix of $D\pi_X(x)$ written with respect to the basis $\tilde{\Bb}(\pi_X(x))$ and $\Bb(\pi_X(x))$ is
\begin{equation*}%\label{jac-proj}
D \pi_X(x) =
\begin{pmatrix}
\frac{1}{1+d_X(x)\epsilon_x \kappa_1(\pi_X(x))} &0&0\\
 0& \frac{1}{1+d_X(x)\epsilon_x \kappa_2(\pi_X(x))}&0
\end{pmatrix}
\end{equation*}
and $\epsilon_x = \langle \frac{\pi_X(x)-x}{\|\pi_X(x)-x\|},{\bfn}_X(\pi_X(x))\rangle\in \{-1,+1\}$. 

Let now $A \subset \tin$ be a subset of a triangle of $\tri$.  This implies that    the Jacobian of the projection on $X$ restricted to $A$ is given by 
\begin{equation*}
	\det (\rst{D}{A}\pi_X)(x) =\frac{\cos \alpha_x}{(1+d_X(x)\epsilon_x \kappa_1(\pi_X(x)))(1+d_X(x)\epsilon_x \kappa_2(\pi_X(x)))} \quad \forall \,x\in A\,.
\end{equation*}

%parameterization \eqref{tri-surf} by setting\begin{equation}\label{jacobianm}Dm(x)(v) = \cos \alpha_x(I_{|T_{\pi_X(x)}M}  + d_X(x)H(x) D\pi_X(x))v \quad \forall\, x\in \tri\quad \forall v\in T_x\tri\end{equation}where $\alpha_x$ is the angle between the two normals ${\bfn}_X(\pi_X(x))$ and ${\bfn}_\tri(x)$. Note that, in the previous equality, the matrix must be understood as $2\times 2$-matrix acting between  $T_x\tri$ and $T_{\pi_X(x)}M$.

We have $\cos \alpha_x=1+O(\alpha_{\max}^2)$. Moreover, as the principal curvatures are uniformly bounded and $d_X(x)=O(d_{\Hh} (X,\tri))$,  we get 
%$$|\Hh^2(X)-\Hh^2(\tri)| =O(\alpha_{\max}^2 + d_\Hh(X,\tri))\,$$as $\alpha_{\max}^2, d_\Hh(X,\tri)\rightarrow 0$.In particular, as  $det(I+A)=1+Tr(A)+det(A)$, we have 
\begin{equation}\label{determinant}
	\det(\rst{D}{A}\pi_X)(x) = (1+O(\alpha_{\max}^2))(1 + O(d_\Hh(X,\tri))) = 1+O(\alpha_{\max}^2+ d_\Hh(X,\tri))\,
	%\det(\rst{D}{A}\pi_X)(x) = (1+O(\alpha_{\max}^2))(1 + O(h)) = 1+O(\alpha_{\max}^2+ h)\,
\end{equation}
as $\alpha_{\max}^2,d_\Hh(X,\tri) \rightarrow 0$. 
We can conclude by the usual change of variables in the area formula. 
\end{proof}

Because of Lemma \ref{lifting}, let us introduce the following set of assumptions:
\begin{hyp}\label{hyp0l2}

	%Let $(\tri_h)$ be a sequence of triangulation indexed by $h>0$. Assume that, for any $h>0$ the mesh $\tri_h$ is $h$-admissible for $X$ and that
	%\begin{equation*}
%%\max\,\{\alpha_{max}^h,\, d_\Hh(X,\Tt_h)\} = O(h)\,
		%\alpha_{max}^h = O(h)\,
	%\end{equation*}
	%as $h\rightarrow 0$ and where  $\alpha_{max}^h$ denotes the biggest angle between the respective normals to $\tri_h$ and $X$ (see Definition \ref{angle}). 
Let $\{\tri_h\}_h$ be a sequence of triangulations indexed by a parameter $h\rightarrow 0$  such that %. Assume that
\begin{enumerate}[label=$(\roman*)$]
\item %$\diam(\tri_h)\rightarrow 0$ which implies 
 for any $h>0$, the triangulation $\tri_h$ is $h$-admissible for $X$;\label{cond1}
%%\item $\underset{h}{\inf}\; \rm{rect}(\tri_h) \,>\,0\,.$
\item   the sequence $\alpha_{\max}^h = O(h)$ as $h \rightarrow 0$, where $\alpha_{\max}^h$ is the angle defined in  Definition \ref{angle} for $\tri_h$.
\end{enumerate}
\end{hyp}

We point out that, because of  Definition \ref{admis-triang}, we get  $d_\Hh(X,\tri_h) = O(h)$, which implies that, under the assumption of Hypothesis \ref{hyp0l2}, the area error computed in Lemma \ref{lifting} is also $O(h)$. Then, Hypothesis \ref{hyp0l2} implies the geometric consistency of the approximation and helps avoid pathological cases such as the Schwartz lantern.

The following result holds,  generalizing Corollary 5 of \cite{MorvanThibert} to surfaces with smooth boundaries : 

\begin{prop}[{\bf Convergence of the area}]\label{conv-area}
	Let $X$ be a surface satisfying Hypothesis \ref{cond_surf}  and $\{\tri_h\}_h$ a sequence of triangulation satisfying Hypothesis \ref{hyp0l2}. Then we have  $\lim_{h\to 0} \limits\Haus(\tri_h) = \Haus(X)$.
%(\cite{Boris}: Corollary 3, p. 15)
\end{prop}

\begin{proof} The proof follows from Lemma \ref{lifting} and  Definition \ref{admis-triang} \ref{hyp.out}.
\end{proof}

\subsection{From the triangulation to the surface}\label{signal-admi}

This section defines how to carry a signal from the triangulation to the surface.

\begin{defn}[{\bf Projection}] \label{lifting-def}
For every function $f$ defined on an $h$-admissible triangulation $\tri$ for $X$  ($h> 0$) we define the projection of $f$ onto $X$ by
\begin{equation*}f^\ell:X\rightarrow \R\,,\quad\quad
f^\ell(x)
 = f(\pi_X^{-1}(x)) \quad \quad x\in X\,.
\end{equation*}
We point out that the function $f^\ell$ carries on $X$ the signal defined on $\tin$.
\end{defn}

\begin{prop}\label{L1norms}
	Let  $h>0$ and $\tri$ be  and an admissible $h$-triangulation for  $X$.
Then,  for every $f\in W^{1, \infty}(\tri, \R)$, we have
\begin{align*}%\label{estimate-lifting}
\|f^\ell\|_{L^p(X)} &= \|f\|_{L^p(\tin)}+ O(\alpha_{\max}^2 + d_\Hh(X,\tri)),\\
\|\nabla_X f^\ell\|_{L^p(X)} &=  \|\nabla_{\tin} f\|_{L^p(\tin)} + O(\alpha_{\max}^2 + d_\Hh(X,\tri)).
%\|f^\ell\|_{L^p(X)} &= \|f\|_{L^p(\tin)}+ O(\alpha_{\max}^2 + h),\\
%\|\nabla_X f^\ell\|_{L^p(X)} &=  \|\nabla_{\tin} f\|_{L^p(\tin)} + O(\alpha_{\max}^2 + h),
\end{align*}
for every $p\in [1,\infty]$, as $\alpha_{\max}^2, d_{\Hh}(X,\tri)\rightarrow 0$.
\end{prop}

\begin{proof} 
%By separating the two on norm on a $\tin$-part and a $\tout$-part, it suffices to prove that $$\begin{array}{l}\|f^l\|_{L^p(X)}=\|f\|_{L^p(\tin)}+ O(\alpha_{\max}^2 + d_\Hh(X,\tri))\;,\vspace{0.2cm}\\\|\nabla_X f^l\|_{L^p(X)}=  \|\nabla_{\tin} f\|_{L^1(\tin)} + O(\alpha_{\max}^2 + d_\Hh(X,\tri))\,.\end{array}$$ 
The first equality is proved by performing the change of variables $y=\pi_X(x)$ and using Lemma \ref{lifting}. The second relationship is proved by the same arguments and applying the chain rule.
% A rigorous proof is given  in \cite{D88} Lemma 3.
%Moreover, since \eqref{lifting-def} and \cite{D88} (Lemma 3), we have for almost every $x\in\tin$ 
%\begin{equation}\label{projecting}
%\nabla^{\tin} f(x) =  D\pi_X^{-1}(\pi_X(x))\nabla^X f^l(\pi_X(x))
%\end{equation}
%where $D\pi_X$ is defined in \eqref{jac-proj}.  By changing the variable ($y=\pi_X(x)$), the result follows from \eqref{jacobian} and \eqref{determinant}. 
\end{proof}

We point out that, for triangulations verifying Hypothesis \ref{hyp0l2},  the previous proposition generalizes to surfaces with boundary the Sobolev error estimates for Euclidean finite elements.

%%% Local Variables:
%%% TeX-master: "ShapeMatchingBV"
%%% End:

\section{Discretization}\label{Discretization}

This section aims to write down the discretized matching energy to compare the discrete problem defined on triangulations to the continuous one defined on the original surfaces.

For clarity, we detail the definition of all discrete norms and functionals in the framework of finite elements. Although it may seem redundant to a reader accustomed to this kind of formalization, we prefer to define the several operators to prevent misinterpretation of notations.

\subsection{Notations}

Let $\tri$ be a triangulation in the sense of Definition \ref{def.tri}.  We denote by $N_v$, $N_e$, and $N_t$ the number of vertices, edges, and triangles of $\tri$, respectively. The family  $\{v_i\}_{i=1,\dots,N_v}$ denotes the vertices of the triangulation.
%We denoted by $\mathbb{T}_h$ the set of triangulations of $X$ such that every triangle has diameter smaller than $h$.

\begin{figure}[h!]
	\centering
	\begin{tikzpicture}[scale=.8]
	
		\coordinate (X2) at (0,0,0);
		\coordinate (X1) at (5,0,4);
		\coordinate (X3) at (0,0,5);
	
		\coordinate (X12) at (2.5,0,2);
		\coordinate (X13) at (2.5,0,4.5);
		\coordinate (X23) at (0,0,2.5);

		\draw[dashed] (X1) node[right] {$v^k_1$} -- (X23) node[left] {$v^k_{23}$} ;
		\draw[dashed] (X2) node[above] {$v^k_2$} -- (X13) node[below] {$v^k_{13}$} ;
		\draw[dashed] (X3)  node[left] {$v^k_3$}-- (X12) node[above right] {$v^k_{12}$} ;

		\draw[thick,fill=blue,fill opacity=.2] (X1) -- (X2) -- (X3)  -- cycle;
		%\draw[thick,red] (X4)  node[above right] {$v^k_4$}-- (X5) node[below left] {$v^k_5$};

		\node[above right] at (5/3,0,3) {$v_0^k$};

\end{tikzpicture}
	\caption{\small{Labels of various points in the triangle $T_k$.}}
	\label{fig.notations_triangles}
\end{figure}

\noindent For every $k=1,\dots,N_t$, we denote by $\{v_i^k\}_{i=1,2,3}$ the vertices of the triangle $T_k\subset \tri$,  $\left\{ v_{ij}^k \right\}_{1\leq i<j\leq 3}$ the center of the edge linking $v_i^k$ to $v_j^k$,  and $v_0^k = \tfrac 1 3 \sum_{i=1}^3v_i^k$ the center of mass of the triangle (see Figure \ref{fig.notations_triangles}). Analogously we denote by $\{f_i^k\}_{0\leq i\leq 3}$ the values of the function $f$ at location $\{v_i^k\}_{0\leq i\leq 3}$ of $T_k$.

\subsection{$P_0$ and $P_1$ triangular finite elements}

Let us start with the following definition,
\begin{defn}
	For a given triangulation $\tri$, we denote by $\PP_{0}(\tri)$ (resp. $\PP_{1}(\tri)$), the set of functions that are constant on the interior of each triangle and null on their edges (resp. the set of the continuous functions that are affine on each triangle). 
\end{defn}

The elements of  $\PP_{0}(\tri)$ (resp. $\PP_{1}(\tri)$)  are completely described by their values $\{f_0^k\}_k$ at the center of mass $\{v_0^k\}_k$ (resp. their values $\{\{f_i^k\}_{1\leq i\leq 3}\}_{k}$ at vertices $\{\{v_i^k\}_{1\leq i\leq 3}\}_{k}$) of the triangulation. Note that for each triangle $T_k\in\tri$
\begin{equation*}
\forall \,f\in \PP_{1}(\tri)\,, \quad  \,f_0^k=f(v_0^k)=f(\sum_{i=1}^3v_i^k/3)=\sum_{i=1}^3 f(v_i^k)/3=\sum_{i=1}^3 f_i^k/3\,.%\label{eq:center-value}
\end{equation*}
On the other hand, if $f\in \PP_{0}(\tri)$, as $f$ is null on the edges of each triangle, we can not compute $f_0^k$ by the values at the vertices, and we can only set
\[\forall \,f\in \PP_{0}(\tri)\,, \quad f_0^k = f(v_0^k)\,.\]
We denote by $p_0$ the $L^2$ projection of $\PP_{1}(\tri)$ on $\PP_{0}(\tri)$ where for $f\in \PP_{1}(\tri)$, the function $p_0(f)$ is the unique element of $\PP_{0}(\tri)$ such that
\begin{equation}
 \label{eq:proj}
  p_0(f)(v_0^k)=f(v_0^k)\,.
\end{equation}
In other words, the operator $p_0$ replaces the affine approximation of a signal on each triangle with a constant approximation using the value at the center of mass.

A basis for $\PP_{1}(\tri)$  (barycentric basis) is given by the family  $\{\varphi\}_{i=1,\cdots,N_v}$ with $\varphi_i\in \PP_{1}(\tri)$ and $\varphi_i(v_j)=\delta_{ij}$ (Kronecker's delta), for every $i,j=1,\cdots,N_v$.  Then, every $f\in \PP_{1}(\tri)$ can be written as 
\begin{equation*}\label{discrete-fct}
\forall x\in \tri \quad\quad f(x) = \sum_{j=1}^{N_v} f_j\varphi_j(x)\,,\quad f_j= f(v_j)\,.
\end{equation*}
Remark that there exists a bijection between $\PP_{1}(\tri)$ and $\R^{N_v}$, defined by the following operator
\begin{equation}\label{bijection}	
	P_1\; :(f_1,\dots,f_{N_v})\in \R^{N_v}	\;\;\mapsto\;\; \; f = \sum_{j=1}^{N_v}f_j \, \varphi_j\in \PP_{1}(\tri)\,.
\end{equation}

\subsection{Discrete operators for finite elements} 

For every $k = 1,\cdots,N_t$ the area of the triangle $T_k$ is denoted by $\lvert T_k\rvert$ and is equal to $\frac{1}{2}\lVert \bfn_{T_k}\rVert$ where ${\bfn}_{T_k}\,=\,(v_2^k-v_1^k)\wedge (v_3^k-v_1^k)$.  

\subsubsection{Discrete functional norms}

Depending on how the continuous signal is discretized, various methods may be used to compute the norm of the discrete signal.

\paragraph{$L^p$ norm of $P_0$ finite elements.}
Let $p \geq  1$ and  $f:\tri \to \R$ a function in $L^p(\tri,\R)$.   
The $p$th power of the discrete $L^p$ norm of $f$ is simply defined as
\begin{equation}\label{lp-p0}
	L_{0}^p[f,\tri] = \sum_{k=1}^{N_t} |T_k| |f_0^k|^p.
\end{equation}
The formula \eqref{lp-p0} is exact for signals that are (almost everywhere) constant on each triangle so that we have $L_{0}^p[f,\tri] = \|f\|_{L^p(\tri)}^p$ for any $f\in\PP_{0}(\tri)$.
\paragraph{$L^p$ norm of $P_1$ finite elements.}
As in this work, we are concerned by $L^p$ norms with $p=1,2$, we need to define the discrete operators via an exact formula on piecewise quadratic polynomials. We use the following approximation by the evaluation at the midpoints (see \cite{Allaire} p. 178--179):
\begin{equation}\label{eq.NC}
	%L^p_{1}[f,\tri] =  \frac 1 3 \sum_{k=1}^{N_t} |T_k| (\abs{f( v_{12}^k)}^p + \abs{f ( v_{23}^{k })}^p + \abs{f( v_{31}^k)  }^p  ) 
	L^p_1[f,\tri] =  \frac 1 3 \sum_{k=1}^{N_t} |T_k| (\abs{f^k_{12}}^p + \abs{f^k_{13} }^p + \abs{f^k_{23} }^p  ).
\end{equation}
where for a piecewise linear signal $f\in\PP_1(\tri)$ we have $f^k_{ij} = \frac 1 2 (f^k_{i} + f^k_{j})$ for any $k=1,\dots,N_t$ and $1\leq i<j\leq 3$. 

The formula \eqref{eq.NC} is exact if $p=2$ (\ie{}, $L^2_1[f,\tri] = \|f\|_{L^2(\tri)}^2$) since the function $\rst{f^2}{T_k}$ is a polynomial of degree 2 for any $k=1,\dots,N_t$.
The formula \eqref{eq.NC} is also exact if $p=1$ and $f$ has constant sign on each triangle (\ie{} we have  $L^1_1[f,\tri] = \|f\|_{L^1(\tri)}$ if $f \geq 0$ or $f \leq 0$). % as .$\rst{f^2}{T_k}$ is a polynomial of degree 1. 
Suppose the signal $f\in\PP_1(\tri)$ has a changing sign on a given triangle $T_k$. In that case, the function to integrate is not a polynomial on the entire triangle, and the formula is not exact anymore. Then, the triangle is decomposed into several sub-triangles with constant sign signal and the operator in \eqref{eq.NC} is computed as the sum of the respective operators computed (exactly) on each subtriangle. This is equivalent to performing a local triangulation refinement to ensure the signal has a constant sign on each triangle.

We suppose now that $f\in \PP_{1}(\tri)$ and we define the discrete operators corresponding to the $H^1$ and $BV$ norms. For every $T_k\in \tri$ and for every $f\in \PP_{1}(\tri)$, the gradient of $f$ on $T_k$ can be computed by
\begin{equation}\label{eq.graddis}
	%[\nabla_{\tri} f]_{T_k}= \frac{\bfn_{T_k}}{\| \bfn_{T_k}\|^2}\wedge \left(f_1^k e_1^k+f_2^k e_2^k+f_3^k e_3^k\right )
	[\nabla_{\tri} f]_{T_k}= \frac{ e_2^k \wedge e_3^k }{ \| e_2^k \wedge e_3^k \|^2}\wedge \left(f_1^k e_1^k+f_2^k e_2^k+f_3^k e_3^k\right )
\end{equation}
where 
\[e_1^k=v_3^k-v_2^k\,,\quad e_2^k=v_1^k-v_3^k\,,\quad e_3^k=v_2^k-v_1^k\,. \]
%and
%\[ {\bfn}_{T_k}\,=\,(v_2^k-v_1^k)\wedge (v_3^k-v_1^k)\,.\]
The previous relationship can be stated by writing $f$ in the barycentric coordinates system and recalling that the gradients of the basis elements are perpendicular to the triangle edges.

In this framework, the gradient $\nabla_{\tri} f $ is constant on each triangle and, by convention, is null on the edges.
We can now define the discrete operators of the gradient norms that are exact for every $f \in \mathbb P_{1}$.
\paragraph{Total variation.} 
The total variation of $f \in \mathbb P_{1}$ on $\tri$ is given by 
\begin{equation}\label{discrete-var}
V[f,\tri]= \sum_{k=1}^{N_t} |T_k|\big\|[\nabla_{\tri} f]_{T_k} \big\|\,.
\end{equation}

\paragraph{$H^1$ norm.} For $f\in \PP_{1}(\tri)$,  the square of the  $L^2$ norm of the gradient is given by 
\begin{equation}\label{h1}
H[f,\tri]= \sum_{k=1}^{N_t} |T_k|\big\|[\nabla_{\tri} f]_{T_k} \big\|^2\,\,.
\end{equation}
Remark that $H[f,\tri] = L^2_0[\|\nabla_{\tri} f\|,\tri ]$  as $f\in \mathbb{P}_1(\tri)$ and by formula \eqref{eq.graddis} we have  $\|\nabla_{\tri} f \| \in \PP_0(\tri)$.

\def\mudis{\boldsymbol{\mu}}

\subsubsection{Discrete fvarifold norm} 

The definition of discrete fvarifolds can be posed by considering a fvarifold associated to 
 $(\tri,f)$. We point out  $(\tri,f)$ does not verify the conditions (smoothness) defining fshapes;  however, the fshape definition given in Section \ref{fvarifold} remains consistent with the properties of $\tri$. 
 
Then, with an abuse of language, we call $(\tri,f)$ discrete fshape, and the related fvarifold supported on it is defined as follows :

\begin{equation}\label{rectif-tri}
\var{\tri}{f} = \Haus \res \tri \otimes \delta_{T_\tri(x)} \otimes  \delta_{f(x)}\,.
\end{equation}
%where $\theta$ is called multiplicity and is 1 in the interior of each triangle and 0 on each edge.

Moreover, we approximate such a fvarifold by a discrete operator $\vardis{\tri}{f}$ to simplify the computation. Such an approximation is set in the same way in both $P_0$ and $P_1$ framework :
\begin{equation}
\vardis{\tri}{f} = \sum_{k=1}^{N_t} |T_k| \delta_{(v_0^k, V_k, f_0^k)}\label{eq:approx_fvar}
\end{equation}
where 
\[ v_0^k= \frac{1}{3}(v_1^k+v_2^k+v_3^k),\  f_0^k=f(v_0^k), \  V_k =  \text{Span}\{v_2^k-v_1^k,v_3^k-v_1^k\}\,.\]
The discrete fvarifold norm  is defined by 
\begin{equation}\label{discr-var}
\Var[\var{\tri}{f}] =  \normvar{\vardis{\tri}{f}} \,.
\end{equation}
%which can be easily computed by \eqref{var-prod}. 
We point out that by the formula \eqref{eq:proj}, we have
\[\vardis{\tri}{f}=\vardis{\tri}{p_0(f)}\quad \quad \forall\, f\in \PP_{1}(\tri) \,.\]

We end this section with a technical lemma that will be useful in the following section. It proves an error estimate, with respect to the fvarifold norm, between a discrete signal $f$ (belonging to $P_0$ or $ P_1$) and its projection $f^l$ on the surface (see Definition  \ref{lifting-def}).

\begin{lem}\label{approx-var} Let $X$ be a surface verifying Hypothesis \ref{cond_surf} and $\tri$ a $h$-admissible triangulation for $X$ verifying Hypothesis \ref{hyp0l2}. Then we have 
\begin{equation}\label{bound-var-P0}
\sup_{f\in \PP_{0}(\tri)}\normvar{\vardis{\tri}{f}-\var{X}{f^\ell}}=O(h)\,,
\end{equation}
and
\begin{equation}\label{bound-var-P1}\forall\, f\in \PP_{1}(\tri)\quad \normvar{\vardis{\tri}{f}-\var{X}{f^\ell}}=O(h)(1+\|\nabla f\|_{L^1(\tri)})\,.
\end{equation}

\end{lem}

\begin{proof}
	Let $f\in \PP_{0}(\tri)$. By the change of variables $x=\pi_X(y)$, because of Hypothesis \ref{hyp0l2} and  formula \eqref{determinant}, for every  $f\in\PP_{0}(\tri)$ and $\varphi \in \test$, we get
\[
	\bigg |\int_X \varphi(x,T_x X,f^\ell(x)) d \Haus(x) -\int_{\tin} \varphi(y,T_y \tin, f(y)) d \Haus(y)\bigg |\leq \|\varphi\|_{L^{\infty }(\R^3\times G(3,2)\times \R)}\, O(h) \,. 
\]
Moreover, because of Definition \ref{admis-triang}\ref{hyp.out} and Lemma \ref{null-var}, $\var{\tout}{f}$ converges towards the null fvarifold as $h\to 0$.
Then
\[
\bigg|\int \varphi d(\var{\tri}{f}-\var{X}{f^\ell})\bigg|\leq \|\varphi\|_{L^\infty(\R^3\times G(3,2)\times \R)}O(h)\,.
\]

Now, if $f\in \PP_{0}(\tri)$, then  $f=f_0^{k}$ on the interior of every triangle $T_k$. Then, for every function  $\varphi\in \test$, we get
\[
\begin{array}{ll}
|(\var{\tri}{f}-\vardis{\tri}{f})(\varphi)| &\leq  \displaystyle{\sum_{k=1}^{N_t} \int_{T_k}| \varphi(x,T_x\tri,f(x))-\varphi(v_0^{k}, V_k,f_0^{k}) |\,d\Haus(x)}\\
&\displaystyle{\leq \sum_{k=1}^{N_t} \|\varphi\|_{C^1(\R^3\times G(3,2)\times \R)}\int_{T_k}\|(x,V_k,f(x))-(v_0^{k}, V_k, f_0^{k})\|_{\R^3\times G(3,2)\times \R}}\,d\Haus(x)\\
&\displaystyle{\leq  \sum_{k=1}^{N_t} |T_k|\|\varphi\|_{C^1(\R^3\times G(3,2)\times \R)}\mbox{diam}(T_k)\leq \Haus(\tri)\|\varphi\|_{C^1(\R^3\times G(3,2)\times \R)}O(h)\,.}
\end{array}
\]
where $N_t$ denotes the number of triangles contained in  $\tri$ and $V_k$ the tangent space to $T_k$. So, by \eqref{dual-emb}, we obtain
\[
	\normvar{\var{\tri}{f}-\vardis{\tri}{f}}
	%\leq C\normvar{\var{\tri}{f}-\vardis{\tri}{f}\|_{(\test)'} 
	\leq O(h) \Haus(\tri)\,,
\]
and, by the triangle inequality, we get $\normvar{\mu_{(X,f^\ell)}-\vardis{\tri}{f}}\leq O(h)$ which proves \eqref{bound-var-P0}.

If $f\in \PP_{1}(\tri)$ the proof is similar. The bound for $|(\var{\tri}{f}-\vardis{\tri}{f})(\varphi)|$ depends on $f$ because 
\[
\forall\, x\in T_k\quad |f(x)-f_0^{k}|\leq \|\nabla f\|_{L^1(T_k)}\mbox{diam}(T_k)
\]
and, by the same arguments, we get
$$
%\left|\Var_h[\var{\tri_h}{f}]-\normvar{\var{X}{f^\ell}}^2\right|= 
\normvar{\vardis{\tri}{f}-\var{X}{f^\ell}}\leq  O(h)(1+\|\nabla f\|_{L^1(\tri)})\,
$$
which proves \eqref{bound-var-P1}
 
\end{proof}

\subsection{Discretization of continuous signals}\label{part.disc}

Finally, we present the discretization process (on admissible triangulations) of a signal defined on a surface. 
Let $X$ be a surface verifying Hypothesis \ref{cond_surf} and $\tri$ be a $h$-admissible triangulation for $X$ (for some $h>0$), and $f\in BV(X)$ (resp.  $H^1(X)$, $L^2(X)$). 

According to  Definition \ref{admis-triang}\ref{hyp.ext}, by using the map $\mbox{Ext}$ defined in \eqref{ext}, we can extend the manifold $X$ to a larger suitable manifold $\xeta$ such that $\tri \subset \Nc_h(\xeta)$. Moreover, because of Theorem \ref{approx2}, the signal $f$ can be extended to a signal $\tilde f$ defined on $\xeta$ which is $W^{1,\infty}(\xeta)$ and with a small norm on $\xeta\setminus X^{-\eta}$. 

Let $T$ be a triangle contained in $\tri$, we set the following discretization process: 
\begin{itemize}
\item  {\bf $P_0$ elements}: as every $P_0$ element is null on each edge of $T$, we define the piecewise (a.e.) constant $f_h$ corresponding to $f$ on  $T$ as 
\[
f_h=
\begin{cases}
0 & \mbox{on the edges of $T$}\,,\\
\displaystyle{\tfrac{1}{3} \left( \tilde f(\pi_{\xeta}(v_1))+\tilde f(\pi_{\xeta}(v_2))+\tilde f(\pi_{\xeta}(v_3))\right)} &  \mbox{otherwise,}\\
\end{cases}
\] 
where $v_1,v_2,v_3$ denote the three vertices of $T$.

\item {\bf $P_1$ elements:} the piecewise linear $f_h$ corresponding to $f$ on $T$ is defined as 
\[
f_h=P_1\big(\tilde f(\pi_{\xeta}(v_1)),\tilde f(\pi_{\xeta}(v_2)),\tilde f(\pi_{\xeta}(v_3))\big)
\] 
where $v_1,v_2,v_3$ denote the three vertices of $T$ and $P_1$ is defined by \eqref{bijection}. 
\end{itemize}

This process defines a discrete signal $f_h$ on the triangulation whose different norms can be computed by the operators introduced above.
We note that these operators represent an approximation of their continuous counterparts, which motivates the study of the relationship between the discrete and continuous minimizers discussed in the next section.

\section{Discrete problem and $\Gamma$-convergence results}\label{GammaConvSection}

This section aims to show the $\Gamma$-convergence of the discretized problems to the continuous one when the triangulation is fine enough. 
The main consequence is that the discrete optimal solution computed on triangulations represents a good approximation of the optimum computed on the surface (Theorem \ref{minima}). The proof is provided under the geometric assumptions established in Hypothesis \ref{hyp0l2}. This implies, in particular, that the area error (see Lemma \ref{lifting}) is $O(h)$, ensuring the geometric consistency of the approximation process.

Let $X$, $Y$ be two surfaces verifying  Hypothesis \ref{cond_surf}, and  $g\in BV(Y)$. We denote by $\{\tri_h\}_h$, $\{\Yy_h\}_h$  two sequences of $h$-admissible triangulations of $X$  and $Y$, respectively, verifying Hypothesis \ref{hyp0l2}. We recall that the control on the angle between continuous and discrete corresponding normal vectors prevents pathological cases like the Schwartz lantern. We finally denote by $\{g_h\}_h$ the discretization of $g$ on the sequence of triangulations  $\{\Yy_h\}_h$ obtained by the discretization process described in Section \ref{part.disc}.

We point out that $h$ denotes the family parameter and the triangulation diameter (the maximum diameter of the triangles contained in $\tri_h$). Then, the $\Gamma$-convergence results hold for the typical case of increasingly fine triangulations. Finally, we note that the regularity condition \eqref{regular-tri} is fundamental in our proof to ensure global error estimate in the finite elements framework.

Then, for every $h$, the discrete energy is defined by
$$ E_h:\PP_{1}(\tri_h)\rightarrow \R\cup\{+\infty\}$$ 
\begin{equation}\label{eq.EBV}
	E_h(f_h)= \left(L_1^1[f_h,\tri_h] + V[f_h,\tri_h])\right) +  
	%\normvar{\vardis{\tri_h}{f}-\vardis{\Yy_h}{ g_h}}^2,
	\Var[\mu_{(\tri_h,f_{h})} - \mu_{(\Yy_h,g_h)}]^2,
\end{equation}
where the  operators $L_1^1$, $V$, $\mbox{Var}$ are define in \eqref{eq.NC}, \eqref{discrete-var}, and \eqref{discr-var}, respectively. For every $h>0$, the optimal signal on $\tri_h$ is defined by the following problem 
\begin{equation}\label{discreteP}
\inf_{f_h\in \PP_{1}(\tri_h)} E_h(f_ h).
\end{equation}
The main goal of this section is to prove that the minimizers of $E_h$ converge to the minimizers of the continuous problem 
\begin{equation}\label{Fbv2}
\inf_{f\in BV(X)} E(f), \quad \quad E(f) = \|f\|_{BV(X)} +  \normvar{\var{X}{f}-\var{Y}{g}}^2\,.
\end{equation}
First of all, we prove the existence result for the discrete problem :
\begin{prop}\label{min-discret} For every $h>0$, there exists at least one solution  to problem \eqref{discreteP}.
\end{prop}

\begin{proof} For every $f\in\PP_{1}(\tri_h)$ we have  $\|\nabla_{\tri_h} f\|_{L^1(\tri_h)}=
  V[f,\tri_h]$ and $L_1^1[f,\tri_h] = \|f\|_{L^1(\tri_h)}$. Then, every minimizing sequence is bounded in $BV$ so that it weak-$*$ converges in $BV$ (up to a subsequence) to some function belonging to $\PP_{1}(\tri_h)$. The result follows from Lemma \ref{conv-var} and the lower semicontinuity of the total variation with respect to the weak-$*$ convergence.
\end{proof}
As the discrete and continuous energies are not defined in the same space, we introduce a suitable topology to compare the two spaces and generalize the classical definition of $\Gamma$-convergence (see \cite{Braides}):

\begin{defn}[{\bf $S$-topology and $\Gamma$-convergence}]\label{Gamma}
	Let $X$ be a surface satisfying Hypothesis \ref{cond_surf}, $f\in BV(X)$, and $\{\tri_h\}_h$  be a sequence  of admissible triangulations for $X$ verifying Hypothesis \ref{hyp0l2}. In this definition, $\{f_h\}_h$ denotes a sequence of functions such that $f_h\in \PP_{1}(\tri_h)$ for every $h>0$. 

	\begin{itemize}
		\item 	We say that $\{f_h\}_h$ converges to $f$ with respect to the  $S$-topology ($f_h \overset{S}{\rightharpoonup} f$) if and only if 
			\begin{equation}\label{S-conv-bv}
				\lim_{h\to 0} \|{f_h^\ell - f}\|_{L^1(X)}
				%f_h^\ell\rightarrow f \text{ in  $L^1(X)$,}  
%\quad\underset{h}{\sup}\, \|f_h\|_{BV(\tout_h)}=O(h)
			\end{equation}
			where for any $h>0$, $f_h^\ell$ is the projection of $f_h$ onto $X$ (see Definition \ref{lifting-def}). 
		\item 	We say that $(E_h)_h$  $\Gamma$-converges to $E$ if the following conditions hold: 
			\begin{enumerate}[label=$(\roman*)$]
				\item {\bf Lower bound:} for every $f\in BV(X)$ and for every sequence  $\{f_h\}_h$ %with $f_h\in \PP_{1}(\tri_h)$ for every $h>0$ %where $\{\tri_h\in \mathbb{T}_h\}$ verifies \eqref{hyp0l2} and 
					such that $f_h \overset{S}{\rightharpoonup} f$, we have 
					\[E(f)\leq \liminf_{h\rightarrow 0} E_h(f_h)\,; \]
				\item {\bf Upper bound:} for every $f\in BV(X)$ there exists  a sequence  $\{f_h\}_h$ %with $f_h\in  \PP_{1}(\tri_h)$ for every $h>0$ %where $\{\tri_h\in \mathbb{T}_h\}$ verifies \eqref{hyp0l2}, 
					such that $f_h \overset{S}{\rightharpoonup} f$ and
			\[E(f)\geq \limsup_{h\rightarrow 0} E_h(f_h)\,.\]
			\end{enumerate}
	\end{itemize}
\end{defn}

\noindent We can now prove the main result of this work :

\begin{thm}\label{gamma-conv} The sequence $\{E_h\}_h$ \eqref{eq.EBV} $\Gamma$-converges to $E$ \eqref{Fbv2} with respect to the $S$-topology of Definition \ref{Gamma}.
\end{thm}

\begin{proof} {\bf Lower bound}.  Let $\{f_h\}_h$ be a sequence of functions such that $f_h\in \mathbb{P}_{1}(\tri_h)$ for every $h>0$ and  $f_h \overset{S}{\rightharpoonup} f \in BV(X)$.  Thus we have 
	\[ f_h^\ell  \overset{L^1(X)}{\longrightarrow} f\]
where $f_h^\ell$ is the projection of $f_h$ onto $X$ (see Definition \ref{lifting-def}). 
Without loss of generality, we can suppose that 
\begin{equation}\label{bound-energy}
	\sup_{h>0} E_h(f_h)<\infty\,.
\end{equation}
By  Proposition \ref{L1norms} and Hypothesis \ref{hyp0l2}
we get  
\begin{equation}\label{bv-discrete-cont}
\|f^\ell_h\|_{BV(X)}\leq \|f_h\|_{BV(\tri_h)} + O(h) = L_1^1[f_h,\tri_h] + V[f_h,\tri_h] + O(h).
\end{equation}

\noindent Moreover, by Lemma \ref{approx-var}, we have
$$  \normvar{\vardis{\tri_h}{f_h}-\var{X}{f_h^\ell}}^2\overset{h\rightarrow 0}{\longrightarrow} 0\,$$
so that, as $f^\ell_h \overset{L^1(X)}{\rightarrow} f$, Lemma \ref{conv-var} implies 
$$\normvar{\var{X}{f_h^l}-\var{X}{f}}^2\overset{h\rightarrow 0}{\longrightarrow} 0$$
then
\begin{equation}\label{conv-var-bv-3}
	\Var[\var{\tri_h}{f_h}-\var{\Yy_h}{ g_h}]^2 =\normvar{\vardis{\tri_h}{f_h}-\vardis{\Yy_h}{ g_h}}^2\overset{h\rightarrow 0}{\longrightarrow} \normvar{\var{X}{f}-\var{Y}{ g}}^2\,.
\end{equation}

\noindent Now, as the $BV$ norm is lower semicontinuous with respect to the $L^1$ convergence, we get
\begin{equation}\label{final-l2}
	E(f)\leq \liminf_{h\rightarrow 0} \left(\|f_h^\ell\|_{BV(X)} +  \Var[\var{\tri_h}{f_h}-\var{\Yy_h}{ g_h}]^2\right)\leq \liminf_{h\rightarrow 0} E_h(f_h). 
\end{equation}

{\bf Upper bound}. Because of Theorem \ref{approx2}, we can assume that $f\in C^1(X)$ and get the general result by a diagonal argument.

The discretization process defines,  for every $h>0$, an extension $\xeta$ of $X$ to extend $f$ to  $\tilde{f}\in W^{1,\infty}(\xeta)$ (Theorem \ref{approx2}) and define the discrete function $f_h$ (see Section \ref{part.disc}). Of course, the extension $\xeta$ can depend on $h$, but in the following, we prove some estimates for a fixed $h>0$ and write $\xeta$. In particular, as $\tilde{f}$ is defined via a reflection symmetry with respect to the boundary of $X$ (see Theorem \ref{approx2}),  we have
\begin{equation}\label{bound-ext}
	\|\tilde{f}\|_{W^{1,\infty}(\xeta)} \leq \|f\|_{W^{1,\infty}(X)}
\end{equation}
which implies that
$f_h \in W^{1,\infty}(\tri_h)$ for every $h$, and $\sup_h \|f_h\|_{W^{1,\infty}(\tri_h)} \leq \|f\|_{W^{1,\infty}(X)}$.

The usual estimates for interpolation error, as $h \rightarrow 0$, give  
\begin{equation}\label{ciarlet-est}
\|\tilde{f}\circ\pi_{\xeta}-f_h\|_{W^{1,1}(\tri_h)} \leq O(h) \|\tilde{f}\circ\pi_{\xeta}\|_{W^{1,\infty}(\tri_h)}\,.
\end{equation}
This can be deduced from  Theorem 3.1.6  in \cite{Ciarlet} (applied with $k=s=0$, $q=m=1$, $p=\infty$) by considering the sum on all the triangles and by using 
%\eqref{regular-tri} coupled with 
the fact that the area of every triangle is bounded by $\pi(h/2)^2$. We note, in particular, that these estimates hold for regular triangulations verifying \eqref{regular-tri}.

Now, according to Theorem \ref{approx2}, $\tilde{f}$ coincides with $f$ on $X^{-\eta}$ and its $BV$ norm can be arbitrarily small on $X \setminus X^{-\eta}$.
Then, we get 
$$
\|f-f_h^\ell\|_{W^{1,1}(X)}=\|\tilde{f}-f_h^\ell\|_{W^{1,1}(X^{-\eta})} + O(h)\,.
$$
Moreover, as $\tri_h$ virifies Definition \ref{admis-triang}\ref{hyp.bij} , $X^{-\eta}$ can be projected on a subset of $\tri_h$, and by using \eqref{bound-ext} and \eqref{ciarlet-est}, we get
$$
\|f-f_h^\ell\|_{W^{1,1}(X)}=O(h)\,.
$$
This proves in particular that  $f_h^\ell\rightarrow f$ strongly in $L^1(X)$, which means that $f_h \overset{S}{\rightharpoonup} f$.
Now, by Proposition \ref{L1norms} and Hypothesis \ref{hyp0l2}, we have
\[
	\|f^\ell_h\|_{L^1(X)} = \|f_h\|_{L^1(\tin_h)}  + O(h)= L_1^1[f_h,\tri_h]  + O(h),
\]
so that
\[
	L_1^1[f_h,\tri_h] \rightarrow \|f\|_{L^1(X)}^1\quad \mbox{as} \quad h\rightarrow 0.
\] 
\noindent Similarly, the convergence of the total variation term follows from
$$  \|\nabla_{X} f_h^\ell\|_{L^1(X)}=\|\nabla_{\tri_h} f_h\|_{L^1(\tin_h)} + O(h) = V[f_h,\tri_h]+O(h)\,.$$
As in the case of the lower bound, the convergence of the fvarifold term follows by Lemmas \ref{approx-var} and \ref{conv-var}, which finally proves
\[E(f)=\lim_{h\rightarrow 0} E_h(f_h). \]
\end{proof}
The $\Gamma$-convergence result implies, in particular, the convergence of the minimizers :

\begin{thm}[{\bf Convergence of minimizers}]\label{minima}
Let $\{\tri_h\}_h$ be a sequence of admissible triangulations for $X$ verifying  Hypothesis \ref{hyp0l2}. Let $\{f_h\}_h$ be a sequence of minimizers of $E_h$ (i.e., $E_h(f_h)=\min_{f\in \PP_{1}(\tri_h)} E_h(f) $). Then, $\{f^\ell_h\}_h$ weakly-$*$ converges in $BV(X)$ (up to a subsequence) to a minimizer of $E$ and
\[
\lim_{h\rightarrow 0}  \min_{f\in \PP_{1}(\tri_h)} E_h(f) =  \min_{f\in BV(X)} E(f). \]
\end{thm}

\begin{proof}

We consider the sequence  $\{f_h \in \PP_{1}(\tri_h)\}_h$  of minimizers of $E_h$ and, without loss of generality, we can also  suppose  that
\[\sup_{h>0} E_h(f_h) < \infty\,.\]
Similarly to \eqref{bv-discrete-cont}, we have 
\[
	\|f^\ell_h\|_{BV(X)}\leq L_1^1[f_h,\tri_h]+ V[f_h,\tri_h] + O(h),
\]
so that  $\{f_h^\ell\}_h$ is uniformly bounded in $BV(X)$. Then, there exists a subsequence (not relabeled) such that  $\{f_h^\ell\}_h$ weak-$*$ converges to $f^{\infty}\in BV(X)$ and, by Lemmas \ref{approx-var} and \ref{conv-var}, we get 
\[
	\Var_h[\var{\tri_h}{f_h}-\var{\Yy_h}{ g_h}]^2 \rightarrow \normvar{\var{X}{f^\infty}-\var{Y}{ g}}^2.
\]
Then, by lower  semicontinuity the $BV$ norm with respect to the $L^1$ topology, we get
\[
	\min_{f\in BV(X)} E(f)  \leq E(f^{\infty})\leq  \liminf_{h\rightarrow 0} \min_{f\in \PP_{1}(\tri_h)}E_h(f).
\]
The other inequality follows from the upper bound condition of $\Gamma$-convergence applied to a minimizer of $E$. 
This proves, in particular, that $f^\infty$ minimizes $E$.
\end{proof}

\begin{rem}[{\bf $H^1$ model}]\label{H1-gamma-conv} The discrete problem for $H^1$ signals is defined  by  the discrete energy
\[
	E_h(f)=L_1^2[f,\tri_h] + H_h[f,\tri_h] + \Var_h[\var{\tri_h}{f}-\var{\Yy_h}{g_h}]^2.
\]
Then, Theorems \ref{gamma-conv} and \ref{minima}  still hold in this case, and their proofs can be adapted by considering as $S$-topology the $L^2$ convergence of the projection (i.e., $f_h\overset{S}{\rightharpoonup} f$ if $f_h^\ell \rightarrow f$ strongly in $L^2(X)$) and using the compact embedding in $L^2(X)$ and the compactness of the unit ball with respect to the weak topology of $H^1(X)$. 

\end{rem}

\begin{rem}[{\bf $L^2$ model}]\label{L2-gamma-conv} 
In this case, the discrete energy is defined on the set of $P_0$ finite elements by the following function to minimize on $\PP_{0}(\tri_h)$:
$$E_h:\PP_{0}(\tri_h)\rightarrow \R\cup\{+\infty\}\;,\quad 
E_h(f_h)= \frac{\gamma_f}{2}L_{0}^2[f_h,\tri_h]  + \frac{\gamma_W}{2}\Var[\mu_{(\tri_h,f_{h})} - \mu_{(\Yy_h,g_h)}]^2\,.$$
%	\label{eq:DiscEnerLtwo}
We must suppose that $\gamma_f/\gamma_W$ is large enough to apply Theorem \ref{l2-ex}. We note that in the case of  $W'$-fvarifold norm (see Remarks \ref{choiceW} and \ref{l2-model-exist}), Lemma \ref{approx-var}  still holds and Lemma \ref{conv-var} can be proved under the hypothesis of $a.e$ convergence of signals.

For  $L^2$ signals, it is sufficient to define the $S$-topology, $f_h\overset{S}{\rightharpoonup} f$, via the following fvarifold convergence
\begin{equation}
\var{X}{f_h^\ell} \overset{*}{\rightharpoonup} \var{X}{f} \; \mbox{in}\; \mathcal{M}^X \,,
% \quad\underset{h}{\sup}\, \|f_h\|_{L^2(\tout_h)}=O(h)
\end{equation}
where $\mathcal{M}^X$ is defined in Definition \ref{MX}. Because of \eqref{en-en}, we get the lower semicontinuity for the lower bound of $\Gamma$-convergence. The upper bound follows similarly to the $BV$ case, knowing that the discretization process guarantees the strong convergence in $L^2$. 

Concerning the convergence of minimizers, Theorem \ref{l2-ex} (in particular \eqref{bound-min}) implies that the sequence of discrete minimizers $\{f_h\}_h$ is uniformly bounded in $L^\infty$. Then, the sequence of measures $\{\var{X}{f_h^\ell}\}_h \subset \mathcal{M}^X$ is tight and, because of the Prokhorov's theorem,  it  weak-$*$ converges (up to a subsequence) to some $\mu_\infty\in \mathcal{M}^X$. 
Now, the lower semicontinuity of $\tilde{E}$ with respect to the weak-$*$ convergence of measures and the fact that $\tilde{E}$ is minimized by a fshape associated to a $L^2$ function (see Remark \ref{l2-model-exist}), we get
\[\min_{f\in L^2(X)}E(f)\leq \tilde{E}(\mu_{\infty})\leq \liminf_{h\rightarrow 0} \tilde{E}(\var{X}{f_h^\ell})\leq \liminf_{h\rightarrow 0} \min_{f\in \PP_{0}(\tri_h)} E_h(f)\,,\]
that gives the needed inequality. The other one follows from the upper bound condition of $\Gamma$-convergence.
In particular, we get 
$$\min_{f\in L^2(X)}\; E(f) =\tilde{E}(\mu_\infty)$$ 
and, from Theorem \ref{l2-ex}, there exists $f_*\in L^2(X)$ such that $\mu_\infty=\mu_{(X,f_*)}$. Thus, $f_h  \overset{S}{\rightharpoonup} f_*$.
\end{rem}

\section*{Appendix: Approximation theorem for $BV$ functions on manifolds}\label{BV_manifold}

\begin{thm-non}%\label{approx1}
  Let $X$ be an orientable Riemannian compact manifold with boundary $\partial X$ and let $X_0=X\setminus \partial X$. For any $f\in BV(X)$ there exists a sequence $\{f_h\}_h\subset C^1(X_0)$ such that 
$$f_h\to \rst{f}{X_0} \text{ in }L^1(X_0,\R)\quad
\text{and}\quad
|D_Xf|(X)=\displaystyle{\lim_{h\to\infty}\int_{X_0}|\nabla f_h| \oX}\,.$$
\end{thm-non}

\begin{proof}
  Let $\{U_i\}_{i\in\llbracket 1,n \rrbracket}$ be a finite atlas on $X_0$ and for any $i\in \llbracket 1,n \rrbracket$ let $\varphi_i:U_i\to V_i$ be a local chart such
  that $\overline{U_i}$ is compact and $\varphi_i$ is the restriction to $U_i$
  of a $C^1$ diffeomorphisms from $\overline{U_i}\to\overline{V_i}$
  (such an atlas exists since $X$ is compact). Let $\{\eta_{j}\}_{j\geq 0}$
  be a partition of unity such that $\text{supp}(\eta_j)$ is compact  for any $j\geq 0$ and there exists a partition
  $\{J_i\}_{i \in\llbracket 1,n \rrbracket}$ of $\mathbb{N}$ for which
  $\text{supp}(\eta_{j})\subset U_i$ for any $j\in J_i$ which is
  locally finite on any $U_i$ (i.e., for any $x\in U_i$, there exists an
  open set $U_i(x)\subset U_i$ such that $\text{supp}(\eta_j)\cap
  U_i(x)=\emptyset$ for any $j\in J_i$ except on a finite number of
  $j$'s).

For any $j\in J_i$ we consider $\epsilon_j$ such that
  $d(\varphi_i(\text{supp}(\eta_j)),V_i^c)>\epsilon_j$ and for
  $\beps=\{\epsilon_j\}_{j\geq 0}$ we consider the linear operator
  $L_\beps:BV(X)\to C^1(X_0)$ defined by
  \[L_\beps f =  \sum_{i=1}^n\varphi_i^*\Big(\sum_{j\in J_i}\psi_i^*(f\eta_j)* \rho_{\epsilon_j}\Big)\]
where $\psi_i:V_i\to U_i$ is the inverse mapping of $\varphi_i$. We recall the classical notation of differential geometry for pullbacks where for any function $\ell\in C_c(U_i)$, $\psi^*_i\ell = \ell\circ\psi_i$ and for any $v\in \chi_c^1(U_i)$, $\psi_i^*v=(d\psi_i)^{-1}v\circ\psi_i$. We recall that $\chi^1_c(X_0)$ denotes the set of $C^1$ vector fields $u:X \to TX$ on $X$ compactly supported in $X_0$.
Eventually, on every $V_i$, we introduce $\alpha_idx$ the pullback of $\rst{\oX}{U_i}$ by $\psi_i$ on $V_i$ such that for any $\ell\in C_c(U_i)$, we have $\int_{U_i}\ell\oX=\int_{V_i}(\psi_i^*\ell) \alpha_i dx$.

Let $\delta>0$. We can assume that for any $1\leq i\leq n$ and any $j\in J_i$, we have $\epsilon_j$ small enough so that
\[\int_{V_i}|\psi_i^*(f\eta_j)*\rho_{\epsilon_j}-\psi_i^*(f\eta_j)|\alpha_idx\leq \delta 2^{-j}\,.\]
Since $\{\eta_j\}_{j\geq 0}$ is a partition of unity, we have $f=\sum_{i=1}^n\limits\sum_{j\in J_i}\limits f\eta_j$ and
\begin{equation}
  \label{eq:14}
  \int_X|L_\beps f -f |\oX\leq \sum_{i=1}^n\int_{V_i}\sum_{j\in V_j}|\psi_i^*(f\eta_j)*\rho_{\epsilon_j}-\psi_i^*(f\eta_j)|\alpha_idx\leq 2\delta\,.
\end{equation}
This first inequality is enough to prove an approximation result in a $L^1$ sense. We turn now to the control of the total variation part.

Let $u \in \chi^1_c(X_0)$. We have the following decomposition using the integration by part formula (\ref{eq:13}) for equality $(a)$ and the classical integration by part on $\mathbb{R}^d$ for equality $(b)$
  \begin{align*}
    \int_XL_\beps f\divx(u)\oX &= \sum_{i=1}^n\int_{X}\varphi_i^*\Big(\sum_{j\in J_i}\psi_i^*(f\eta_j)* \rho_{\epsilon_j}\Big)\divx(u)\oX\\
&\stackrel{(a)}{=}-\sum_{i=1}^n\int_{X}u\bigg(\varphi_i^*\Big(\sum_{j\in J_i}\psi_i^*(f\eta_j)* \rho_{\epsilon_j}\Big)\bigg)\oX\\
&=-\sum_{i=1}^n\int_{V_i}\sum_{j\in J_i}(\psi^*_iu)\Big(\sum_{j\in J_i}\psi_i^*(f\eta_j)* \rho_{\epsilon_j}\Big)\alpha_idx\\
    &\stackrel{(b)}{=}\sum_{i=1}^n\int_{V_i}\sum_{j\in J_i}[\psi_i^*(f\eta_j)]*
    \rho_{\epsilon_j}\mdiv(\alpha_i\psi_i^*u)dx\\
    &=\sum_{i=1}^n\int_{V_i}\sum_{j\in
      J_i}\psi_i^*(f\eta_j)\mdiv([\alpha_i\psi_i^*u]*
    \rho_{\epsilon_j})dx\\
    &=\sum_{i=1}^n\underbrace{\int_{V_i}\sum_{j\in
        J_i}\psi_i^*(f)\mdiv(\psi_i^*\eta_j([\alpha_i\psi_i^*u]*
      \rho_{\epsilon_j}))dx}_{A_{ij}} \\
    & \qquad\qquad-\sum_{i=1}^n\underbrace{\int_{V_i}\sum_{j\in
        J_i}\psi_i^*(f)([\alpha_i\psi_i^*u]*
      \rho_{\epsilon_j})(\psi_i^*\eta_j)dx}_{B_{ij}}
  \end{align*}
with
\begin{equation*}
  \label{eq:2}  
  A_{ij}
  =\int_{V_i}\sum_{j\in
    J_i}\psi_i^*(f)\left[\mdiv\Big(\alpha_i\psi_i^*\eta_j((\psi_i^*u)*
  \rho_{\epsilon_j})\Big)+\mdiv\Big(\psi_i^*\eta_j\big([\alpha_i\psi_i^*u]*
  \rho_{\epsilon_j}-\alpha_i[(\psi_i^*u)*
  \rho_{\epsilon_j}]\big)\Big)\right]dx\,.
\end{equation*}
However, for any $\delta>0$, denoting $| \cdot |_x$ the norm at $x\in X$ induced by the metric, for $x\in\text{supp}(\eta_j)$ and $\epsilon_j$ small enough  we have

\begin{align*}
 % \begin{split}
    |\varphi_i^*((\psi_i^*u)\ast\rho_{\epsilon_j})|_x(x) & =\bigg|(d_x\varphi_i)^{-1}\Big(\int_{V_i}d_{\psi_i(\varphi_i(x)-y)}\varphi_iu(\psi_i(\varphi_i(x)-y))\rho_{\epsilon_j}(y)dy\Big)\bigg|_x\\
    &\leq
    1+\bigg|(d_x\varphi_i)^{-1}\Big(\int_{V_i}(d_{\psi_i(\varphi_i(x)-y)}\varphi_i-d_x\varphi_i)
    u(\psi_i(\varphi_i(x)-y))\rho_{\epsilon_j}(y)dy\Big)\bigg|_x\\
    &\leq 1+\delta
%  \end{split}
\end{align*}
uniformly in $u$ such that
$\|u\|_\infty\leq 1$ and $j\in J_i$. Hence
\begin{equation}
  \label{eq:4}
  \begin{split}
    \sum_{i=1}^n\bigg| \int_{V_i}\sum_{j\in
      J_i}\psi_i^*(f)\mdiv(\alpha_i\psi_i^*\eta_j((\psi_i^*u)*
    \rho_{\epsilon_j}))dx\bigg|&\leq \sum_{i}\int_X\sum_{j\in J_i}\eta_j(1+\delta)d|D_Xf|\\
    &\leq (1+\delta)|D_Xf|(X)\,.
  \end{split}
\end{equation}
Moreover, for $\epsilon_j$ small enough, for  $x\in \text{supp}(\eta_j)$ we can assume 
\begin{align*}
  \bigg|\varphi_i^*&\frac{[\alpha_i\psi_i^*u]*
    \rho_{\epsilon_j}-\alpha_i[(\psi_i^*u)*
    \rho_{\epsilon_j}]}{\alpha_i}\bigg|_x  =\\
&=\bigg|d_x\varphi_i^{-1}\int_{V_i}\frac{\alpha_i(\varphi_i(x)-y)-\alpha_i(\varphi_i(x))}{\alpha_i(\varphi_i(x))}d_{\psi(\varphi_i(x)-y)}\varphi_iu(\psi_i(\varphi_i(x)-y))\rho_{\epsilon_j}(y)dy\bigg|_x\leq \delta
  \end{align*}
 so that
\begin{equation}
  \label{eq:5}
  \begin{split}
    \sum_{i=1}^n&\bigg|\int_{V_i}\sum_{j\in
      J_i}\psi_i^*f\mdiv\bigg(\alpha_i\psi_i^*\eta_j\frac{[\alpha_i\psi_i^*u]*
  \rho_{\epsilon_j}-\alpha_i[(\psi_i^*u)*
  \rho_{\epsilon_j}]}{\alpha_i}\bigg)dx\bigg|\\
    &\leq \sum_{i}\int_X \sum_{j\in
      J_i}\eta_j\delta |D_Xf|\leq \delta|D_Xf|(X)
  \end{split}
\end{equation}
Thus we have
\begin{equation}
  \label{eq:6}
  \sum_{i=1}^nA_{ij}\leq|D_Xf|(X)(1+2\delta)\,.
\end{equation}
Let us consider now the $B_{ij}$'s. We have
\begin{equation}
  \label{eq:7}
  \begin{split}
    B_{ij}&=\int_{V_i}\bigg\langle
    \int_{V_i}(\alpha_i\psi_i^*u)(x-y)\rho_{\epsilon_j}(y)dy,\Big(\psi_i^*f\nabla(\psi_i^*\eta_j)\Big)(x)\bigg\rangle
    dx\\
    &=\int_{V_i}\bigg\langle
    (\alpha_i\psi_i^*u)(x),\int_{V_j}\Big(\psi_i^*f\nabla(\psi_i^*\eta_j)\Big)(x-y)\rho_{\epsilon_j}(y)dy\bigg\rangle
    dx\\
    &=\underbrace{\int_{V_i}\alpha_i\psi_i^*f\psi_i^*u(\psi_i^*\eta_j)dx}_{B^1_{ij}}
    + \underbrace{\int_{V_i}\Big\langle
      (\alpha_i\psi_i^*u)(x),\big(\psi_i^*f\nabla(\psi_i^*\eta_j)\big)*\rho_{\epsilon_j}-\psi_i^*f\nabla(\psi_i^*\eta_j)\Big\rangle
      dx}_{B^2_{ij}}\,.
  \end{split}
\end{equation}
Concerning the $B^1_{ij}$ terms, we have
\begin{equation}
  \label{eq:9}
  \sum_{i=1}^n\sum_{j\in J_j}B^1_{ij}=\sum_{i=1}^n\int_X\sum_{j\in J_j}fu(\eta_j)\oX=\int_Xfu(1)\oX=0.
\end{equation}

For the $B^2_{ij}$ terms, let us notice that
$\sup_{V_i}|\alpha_i\psi_i^*u|<\infty$ uniformly in $u$ (since
$\|u\|_\infty\leq 1$) and
$(\psi_i^*f\nabla(\psi_i^*\eta_j))*\rho_{\epsilon_j}\to
(\psi_i^*f\nabla(\psi_i^*\eta_j))$ in $L^{1}(\mathbb{R}^d,dx)$ so that
for $\epsilon_j$ sufficiently small, we can assume that
$|B^2_{ij}|\leq \delta 2^{-(j+1)}$. Summing along the indices, we get
\begin{equation}
  \bigg|\sum_{i=1}\sum_{j\in J_i}B^2_{ij}\bigg|\leq \delta\label{eq:10}
\end{equation}
and, with (\ref{eq:6}), (\ref{eq:9}) and (\ref{eq:10}), we get
eventually that, for sufficiently small values of the $\epsilon_j$'s, we
have 
\begin{equation}
  \label{eq:8}
  \bigg|\int_XL_\beps f\divx(u)\oX\bigg|\leq |D_Xf|(X)(1+2\delta+\delta^2)+\delta\,,
\end{equation}
uniformly in $u\in \chi^2_c(X_0)$ satisfying $\|u\|_\infty\leq 1$.
Taking the supremum over such $u$, we get
\begin{equation}
  \label{eq:11}
  |D_XL_{\beps}f|(X)\leq |D_Xf|(X)(1+2\delta+\delta^2)+\delta.
\end{equation}
Since $\delta$ is arbitrary, we have shown that there exists a sequence
$(\beps_k)_{k\geq 0}$ such that $L_{\beps_k}f\in C^\infty(X_0)$ and 
\begin{equation}
  \label{eq:12}
  \limsup_{k}|D_XL_{\beps_k}f|(X)\leq |D_Xf|(X)\,.
\end{equation}
%Moreover, since we can assume that $L_{\beps_k}f\to f$ in $L^1(X)$, we get from the lower semi-continuity of the total variation in $L^1$ that $L_{\beps_k}f\to f$ for the $S$-topology.
\end{proof}

\newpage 

\bibliographystyle{siam}  % plain alpha
\bibliography{bibliography}	

\begin{thebibliography}{10}

\bibitem{Alexandrov}
{\sc A.~D. Aleksandrov and V.~A. Zalgaller}, {\em Intrinsic geometry of
  surfaces}, Translation of Mathematical Monographs. AMS, 15 (1967).

\bibitem{Allaire}
{\sc G.~Allaire}, {\em {Numerical analysis and optimization: an introduction to
  mathematical modelling and numerical simulation}}, Numerical Mathematics and
  Scientific Computation, Oxford Univ. Press, New York, NY, 2007.

\bibitem{Allard}
{\sc W.~Allard}, {\em Firts variation of a varifold}, Annals of Math., 95
  (1972), pp.~417--491.

\bibitem{Almgren}
{\sc F.~Almgren}, {\em {Plateau's Problem : An Invitation to Varifold
  Geometry}}, {Student Mathematical Library}, 1966.

\bibitem{AFP}
{\sc L.~Ambrosio, N.~Fusco, and D.~Pallara}, {\em Functions of {B}ounded
  {V}ariation and {F}ree {D}iscontinuity {P}roblems}, 2000.

\bibitem{BV_manifold}
{\sc L.~Ambrosio, R.~Ghezzi, and V.~Magnani}, {\em {BV} functions and sets of
  finite perimeter in sub-{R}iemannian manifolds}, Annales de l'Institut Henri
  Poincaré C, Analyse non linéaire, 32 (2015), pp.~489--517.

\bibitem{Bauer2020ANF}
{\sc M.~Bauer, N.~Charon, P.~Harms, and H.-W. Hsieh}, {\em A numerical
  framework for elastic surface matching, comparison, and interpolation},
  International Journal of Computer Vision, 129 (2020), pp.~2425 -- 2444.

\bibitem{Braides}
{\sc A.~Braides}, {\em $\Gamma$-convergence for Beginners}, Oxford University
  Press, Oxford, 2002.

\bibitem{Bronstein}
{\sc A.~Bronstein, M.~Bronstein, and R.~Kimmel}, {\em Numerical Geometry of
  Non-Rigid Shapes}, Springer Publishing Company, Incorporated, 2008.

\bibitem{carbonaro2007note}
{\sc A.~Carbonaro and G.~Mauceri}, {\em A note on bounded variation and heat
  semigroup on {R}iemannian manifolds}, Bulletin of the Australian Mathematical
  Society, 76 (2007), pp.~155--160.

\bibitem{ABN}
{\sc B.~Charlier, N.~Charon, and A.Trouv\'e}, {\em The fshape framework for the
  variability analysis of functional shapes}, Foundations of Computational
  Mathematics, 17 (2017), p.~287–357.

\bibitem{Roussillon2020}
{\sc N.~Charon, B.~Charlier, J.~Glaunès, P.~Gori, and P.~Roussillon}, {\em
  Fidelity metrics between curves and surfaces: currents, varifolds, and normal
  cycles}, in Riemannian {G}eometric {S}tatistics in {M}edical {I}mage
  {A}nalysis, X.~Pennec, S.~Sommer, and T.~Fletcher, eds., Academic Press,
  2020, pp.~441--477.

\bibitem{Rigid-evol}
{\sc G.~Charpiat, G.Nardi, G.~Peyr\'e, and F.-X. Vialard}, {\em Piecewise rigid
  curve deformation via a {F}insler steepest descent}, Interfaces and {F}ree
  {B}oundary, 18(1) (2016), pp.~1--44.

\bibitem{Ciarlet}
{\sc P.~Ciarlet}, {\em The {F}inite {E}lements {M}ethod for {E}lliptic
  {P}roblems}, {SIAM}, {C}lassics in {A}pplied {M}athematics, 40 (2002).

\bibitem{Federer}
{\sc H.~Federer}, {\em Curvature measures}, Trans. {A}mer. {M}ath. {S}oc., 93
  (1959), pp.~418--491.

\bibitem{Ennio}
{\sc E.~D. Giorgi}, {\em Sulla convergenza di alcune successioni d’integrali
  del tipo dell’area}, Collection of articles dedicated to Mauro Picone on
  the occasion of his ninetieth birthday, Rend. Mat. (6), 8 (1975),
  pp.~277--294.

\bibitem{guneysu2013functions}
{\sc B.~G{\"u}neysu and D.~Pallara}, {\em Functions with bounded variation on a
  class of {R}iemannian manifolds with {R}icci curvature unbounded from below},
  Mathematische Annalen,  (2013), pp.~1--25.

\bibitem{Hebey}
{\sc E.~Hebey}, {\em {Nonlinear Analysis on Manifolds: Sobolev Spaces and
  Inequalities}}, Courant Lect. Notes Math. {\bf 5}, New York University
  Courant Institute of Mathematical Sciences, New York, 1999.

\bibitem{Kaltenmark2017}
{\sc I.~Kaltenmark, B.~Charlier, and N.~Charon}, {\em A general framework for
  curve and surface comparison and registration with oriented varifolds}, in
  2017 IEEE Conference on Computer Vision and Pattern Recognition (CVPR), 2017,
  pp.~4580--4589.

\bibitem{LEE2017570}
{\sc S.~Lee, N.~Charon, B.~Charlier, K.~Popuri, E.~Lebed, M.~V. Sarunic,
  A.~Trouvé, and M.~F. Beg}, {\em Atlas-based shape analysis and
  classification of retinal optical coherence tomography images using the
  functional shape (fshape) framework}, Medical Image Analysis, 35 (2017),
  pp.~570--581.

\bibitem{Aston2020}
{\sc E.~Lila and J.~A.~D. Aston}, {\em Statistical analysis of functions on
  surfaces, with an application to medical imaging}, Journal of the American
  Statistical Association, 115 (2020), pp.~1420--1434.

\bibitem{miranda2007heat}
{\sc M.~Miranda, D.~Pallara, F.~Paronetto, and M.~Preunkert}, {\em Heat
  semigroup and functions of bounded variation on {R}iemannian manifolds},
  Journal f{\"u}r die reine und angewandte Mathematik (Crelles Journal), 2007
  (2007), pp.~99--119.

\bibitem{MorvanThibert}
{\sc J.-M. Morvan and B.~Thibert}, {\em Approximation of the normal vector
  field and the area of a smooth surface}, Discrete Comput Geom, 32 (2004),
  pp.~383--400.

\bibitem{Geod-bv2}
{\sc G.~Nardi, G.~Peyr\'e, and F.-X. Vialard}, {\em Geodesics on {S}hape
  {S}paces with {B}ounded {V}ariation and {S}obolev {M}etrics}, {SIAM}
  {J}ournal on {I}maging {S}ciences, 9(1) (2016), pp.~238--274.

\bibitem{Pennec2019}
{\sc X.~Pennec, S.~Sommer, and T.~Fletcher}, {\em Riemannian {G}eometric
  {S}tatistics in {M}edical {I}mage {A}nalysis}, London: Academic Press, 2019.

\bibitem{Polthier}
{\sc K.~Polthier and M.~Schmies}, {\em Straightest geodesics on polyhedral
  surfaces}, Procceding {SIIGRAPH'06},  (2006), pp.~30--38.

\bibitem{Resh}
{\sc Y.~Reshetnyak}, {\em Geometry {IV}}, Encyclopaedia of Mathematical
  Sciences. Springer Verlag, ch. 1. TwoDimensional Manifolds of Bounded
  Curvature, 70 (1993), pp.~3--164.

\bibitem{Roussillon2016}
{\sc P.~Roussillon and J.~A. Glaun\`{e}s}, {\em Kernel metrics on normal cycles
  and application to curve matching}, SIAM Journal on Imaging Sciences, 9
  (2016), pp.~1991--2038.

\bibitem{Si}
{\sc L.~Simon}, {\em Lectures on {G}eometric {M}easure {T}heory}, Proceedings
  of the Centre for mathematical analysis, Australian National University, 3
  (1983).

\bibitem{Boris}
{\sc B.~Thibert and J.-M. Morvan}, {\em Smooth surface and triangular mesh:
  comparison of the area, the normals and the unfolding}, Symposium on solid
  modelling and applications (SMA),  (2002), pp.~147--158.

\bibitem{Younes}
{\sc L.~Younes}, {\em Shapes and {D}iffeomorphisms}, Springer, {A}pplied
  {M}athematical {S}cience, 171 (2010).

\end{thebibliography}

\end{document}